\documentclass[a4paper,10pt]{article}

\usepackage{amsfonts}
\usepackage{amsmath}
\usepackage{amsthm}

\usepackage{amssymb}
\usepackage{makecell}
\usepackage{authblk}
\usepackage{xfrac}
\usepackage{tikz}
\usepackage{tikz-cd}
\usepackage{xcolor}

\usepackage{pifont}
\usepackage{eso-pic}
\usetikzlibrary{matrix}
\usetikzlibrary{cd}

\usepackage{float}

\newtheorem{teo}{Theorem}[section]
\newtheorem{propo}[teo]{Proposition}
\newtheorem{lema}[teo]{Lemma}
\newtheorem{coro}[teo]{Corollary}
\newtheorem{ej}[teo]{Example}
\newtheorem{defi}[teo]{Definition}

\newtheorem{nota}[teo]{Remark}

\usepackage{latexsym}
\usepackage{cancel}
\usepackage{rawfonts}
\usepackage{mathrsfs}
\usepackage{stmaryrd}

\usepackage{enumerate} 

\usepackage{multicol}

\usepackage{turnstile} 

\usepackage{comment}  

\usepackage
{hyperref}
\newcommand{\mm}{\mathcal} 
\newcommand{\rr}{\mathscr} 
\newcommand{\bb}{\mathbb}  
\newcommand{\ff}{\mathfrak} 

\newcommand{\cuatro}{{\bf 4}}
\newcommand{\nene}{{\bf n}}
\newcommand{\bebe}{{\bf b}}

\newcommand{\logi}{{\bf L}}
\newcommand{\pos}[1]{{\bf Pos#1}}
\newcommand{\axi}[1]{{\bf(Ax#1)}}
\newcommand{\cele}{{\sf CPL}}

\newcommand{\bedun}{{\sf BD}} 

\newcommand{\bede}{{\sf BS4}} 
\newcommand{\elefi}{{\sf LFI1}} 
\newcommand{\bededos}{{\sf BD2}}
\newcommand{\cube}{{\sf QBD2}}
\newcommand{\hachedos}{{\sf HBD2}}
\newcommand{\hachedost}{{\sf HBD2}_T}
\newcommand{\hachecube}{{\sf HQBD2}}

\newcommand{\val}[1]{\mathop{{\sf Val}}(#1)}
\newcommand{\valbede}[1]{\|#1\|_{\cube}^\ff A}

\newcommand{\disdash}[1]{\sststile{#1}{}}
\newcommand{\dismodels}[1]{\sdtstile{#1}{}}
\newcommand{\martillo}[1]{\dststile{#1}{}}

\newcommand{\mabedos}{\disdash\hachedos}
\newcommand{\mabedost}{\disdash\hachedost}
\newcommand{\macube}{\disdash\hachecube}
\newcommand{\macuhen}{\sststile{\hachecube}{C}}
\newcommand{\mobedos}{\dismodels{\sf BD2}}

\newcommand{\mocube}{\dismodels{\sf QBD2}}

\newcommand{\nturn}[1]{\lefteqn{\raisebox{.2ex}{\hspace{.03mm} \big /}}\mathbin{#1}}

\newcommand{\pt}{\forall} 
\newcommand{\ex}{\exists} 
\newcommand{\existe}[1]{\no\pt #1\no\,}

\newcommand{\no}{\neg} 
\newcommand{\imp}{\rightarrow} 
\newcommand{\sii}{\leftrightarrow}
\newcommand{\negage}{\reflectbox{\rotatebox[origin=c]{90}{\ensuremath\neg}}\hspace{.8mm}} 
\newcommand{\nof}[1]{\ensuremath{{\sim}\,#1}}
\newcommand{\copi}[1]{\ensuremath{{\copyright}\,#1}}
\newcommand{\estar}{\ensuremath{\text{\ding{72}}}\,}
\newcommand{\bestar}{\ensuremath{\text{\ding{73}}}\,}

\newcommand{\emepe}{{\bf(MP)}}
\newcommand{\gen}{{\bf(Gen)}}
\newcommand{\for}{{\rm For}}
\newcommand{\ter}{{\rm Ter}}
\newcommand{\sent}{{\rm Sent}}

\newcommand{\bduno}[1]{#1\wedge\copi #1}
\newcommand{\bdebe}[1]{#1\wedge\no #1}
\newcommand{\bdene}[1]{\copi #1\wedge\no\copi #1}
\newcommand{\bdcero}[1]{\no #1\wedge\copi #1}

\newcommand{\brac}[1]{\langle #1\rangle}
\newcommand{\set}[1]{\{#1\}}
\newcommand{\vect}[1]{#1_1,\ldots, #1_n}
\newcommand{\libre}[1]{\mathop{{\sf free}}(#1)}
\newcommand{\ese}[1]{S(\ff #1)}

\newcommand{\sneg}{\ensuremath{{\sim}}}

\usepackage{color}

\begin{document}

\title{On a four-valued logic of formal inconsistency  and formal underterminedness}

\author{Marcelo E. Coniglio}
\affil{\small University of Campinas (UNICAMP), Brazil}
\author{G. T. Gomez--Pereira}
\affil{\small Departamento de Matem\'atica, Universidad Nacional del Sur (UNS)-CONICET, Bah\'ia Blanca, Argentina.}
\author{Mart\'in Figallo}
\affil{\small Departamento de Matem\'atica and Instituto de Matemática (INMABB), Universidad Nacional del Sur (UNS), Bah\'ia Blanca, Argentina.
}

\date{}

\maketitle

\begin{abstract}
Belnap--Dunn's relevance logic, $\bedun$, was designed seeking a suitable logical device for dealing with multiple information sources which sometimes may provide inconsistent and/or incomplete pieces of information. $\bedun$ is a four-valued logic which is both paraconsistent and paracomplete. On the other hand, De and Omori while investigating what classical negation amounts to in a paracomplete and paraconsistent four-valued setting,  proposed the expansion $\bededos$ of the four valued Belnap--Dunn logic by a classical negation.

In this paper, we reintroduce the logic $\bededos$ by means of a primitive {\em weak consistency operator} $\copyright$. This approach allows us to state in a direct way that this is not only a Logic of Formal Inconsistency ({\bf LFI}) but also a Logic of Formal Underterminedness ({\bf LFU}). After presenting a natural Hilbert-style characterization of $\bededos$ obtained by means of twist-structures semantics, we propose a first-order version of $\bededos$ called $\cube$, with semantics based on an appropriate notion of partial structures. We show that in $\cube$, $\ex$ and $\pt$ are interdefinable in terms of the paracomplete and paraconsistent negation, and not by means of the the classical negation. Finally, a  Hilbert-style calculus  for  $\cube$ is presented, proving the corresponding and soundness and completeness theorems.

\end{abstract}

\section{Introduction}




In this paper, we are interested in an specific expansion of the four-valued logic known as Belnap-Dunn logic ($\bedun$ hereafter). $\bedun$ was originally developed by M. Dunn, and later stepped further to apply the logic to computer science by N. D. Belnap.
This is a logical system that is well--known for its many applications in different fields such as the development of languages allowing self-reference, semantics of logic programming and, mainly, it is a basic tool in the area of
{\em intelligent} database management or question-answering systems. Databases, especially large ones, have a great propensity to become inconsistent and/or incomplete: first, the information stored is usually obtained from different sources which might conflict with each other; second, the information obtained from each source, even if it is not obviously inconsistent, may hide contradictions.

In order to deal with this situation,  Belnap proposed his logic on four non-classical epistemic truth--values:  $1$ (true) and not false, $0$ (false and not true), these values are
to some extent identifiable with the classical ones, $\nene$ (neither true nor false), the
well-known "undetermined" value of some three-valued logics, and $\bebe$ (both true
and false) also called "overdetermined", the value corresponding to the situation
where several (probably independent) sources assign a different classical value to a
sentence. 

Let us denote by $\bf 4$ the set of truth values $\{1,\bebe, \nene, 0\}$ and consider the language $\{\vee, \wedge, \neg\}$. In a sense, we can say that the lattice $\mathcal{BD}=\langle \cuatro, \{\vee, \wedge, \neg\}\rangle$ given by 

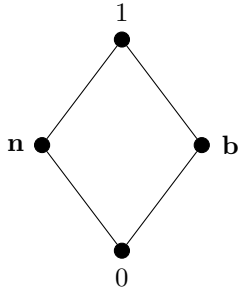
\begin{center}
\begin{tikzpicture}[scale=.7]
\tikzstyle{every node}=[draw,circle,fill=black,inner sep=2pt]
  \node (one) at (0,4) [label=above:$1$] {};
  \node (b) at (-1.5,2) [label=left:$\bf n$] {};
   \node (a) at (1.5,2) [label=right:$\bf b$] {};
  \node (zero) at (0,0) [label=below:$ 0$] {};
  \draw (zero) -- (a) -- (one) -- (b) -- (zero);
\end{tikzpicture} \hspace{2cm}
\end{center}

and where $\neg 0=1$, $\neg 1=0$, $\neg \bebe =\nene$ and $\neg \nene= \bebe$, is an algebraic counterpart of $\bedun$. As it is well--known, we can consider two kinds of orders, truth order and information order, which is an interesting aspect of $\bedun$, and therefore, we can think of the four-element Belnap lattice as a bilattice.

Note here that the designated values are $1$ (truth only) and $\bebe$ (both truth
and falsity), and that $\neg$ is a paraconsistent negation. The values $0$ and $\nene$
are to be taken as falsity only and neither truth nor falsity, respectively. Thus, when
we speak of a sentence being true, we mean it takes either the value $1$ or $\bebe$,
and when we speak of a sentence being false, we mean it takes either the
value $\bebe$ or $0$. Indeed we take there to be only two genuine truth values, truth
and falsity, that are neither exhaustive nor exclusive. Thus, for instance, by
the assignment of the value $\bebe$ to $A$ we are to understand that $A$ is related
to both truth and falsity, not that there is some further truth value, both truth-
and-falsity, in relation to which $A$ stands. 

\

An interesting expansion of \bedun\, was proposed in \cite{FR2} by Font and Rius, namely the {\em Tetravalent Modal Logic} (${\cal TML}$ for short). In their work, these authors devoted to the study of the logic that preserves degrees of truth w. r. t. the class of tetravalent modal algebras (introduced by A. Monteiro and studied by others). The resulting logic turns out to be an extension of \bedun \, by means of a modal operator $\square$.  ${\cal TML}$ was later studied by Coniglio and M. Figallo in \cite{CF, MFThesis} under the perspective of paraconsistent logics. Recently, in \cite{MF} it was proposed a cut-free sequent calculus for ${\cal TML}$ as well as a natural deduction system for it with normalization of proofs.

On the other hand, De and Omori investigated the notion of classical negation from a non-classical perspective in \cite{DeOm}.
In particular, they aim to determine what classical negation amounts to in a paracomplete and paraconsistent four-valued setting. 
They considered different negations for $\bedun$ being each of them  ``classical'' in some respect. These authors conclude that the Boolean complementation (on the lattice $\mathcal{BD}$) is the only negation in this four--valued setting that fulfill all conditions they think a classical negation should verify. Finally, they gave a general semantic characterization of classical negation  consider different expansions of four valued Belnap--Dunn logic by classical negation. 

However, one of the negations considered in \cite{DeOm}, denoted $\neg_2$, is the one defined by $\neg_2 1=0=\neg_2 \bebe$, $\neg_2 \nene= \bebe$ and $\neg_2 0 =1$, and then $\bededos$ is defined as the expansion of $\bedun$ by this negation. In this work, we propose an expansion of $\bedun$ by means of a {\em weak consistency operator} which we denote by \ $\copi$. The resulting logic will be not only a {\em Logic of Formal Inconsistency} (LFI) but also a {\em Logic of Formal Underterminedness} (LFU). Moreover, this expansion turns out to be equivalent to the logic $\bededos$, but our approach will be in the context of the theory of LFI's and LFU's. 

\section{Preliminaries}   

Let $\rr L$ be a (first-degree or propositional) language over a given signature $\Theta$ and let $\for(\Theta)$ (or simply $\for$) be the set of all  well-formed formulas (or simply formulas) over $\rr L$. 
In this work, a logic $\logi$ is a pair $\logi = \brac{\for,\Vdash}$ where $\Vdash$ is a subset of  $2^\for\times \for$. As usual, lowercase Greek letters stand for formulas and uppercase Greek letters stand for sets of formulas. Besides, we shall write $\Gamma, \psi \Vdash \varphi$ instead of $\Gamma\cup\{\psi\}\Vdash\varphi$ and $\psi_1, \dots, \psi_n \Vdash \varphi$ instead of $\{\psi_1, \dots, \psi_n\} \Vdash \varphi$.

Recall that a logic $\logi$ \ is {\em Tarskian} if it satisfies: (1)  if $\alpha\in \Gamma$, then $\Gamma\Vdash\alpha$; (2) if $\Gamma\Vdash\alpha$ and $\Gamma\subseteq \Delta$, then $\Delta\Vdash \alpha$; and (3) if $\Delta\Vdash\alpha$ and $\Gamma\Vdash\beta$ for all  $\beta \in\Delta$, then $\Gamma\Vdash\alpha$. Besides, a logic $\logi$ is {\em compact} (or {\em finitary)} if it holds: 

\

{\bf (TDC)} \ \  if $\Gamma\Vdash\alpha$, there is a finite set  $\Gamma_0\subseteq \Gamma$ such that $\Gamma_0\Vdash\alpha$.

\

Given the logic $\logi=\langle\for, \Vdash\rangle$, a {\em theory} of $\logi$ is any subset of $\for$. A theory $\Gamma$ is {\em closed} if: $\Gamma \Vdash \alpha$ iff  $\alpha \in \Gamma$. A theory $\Gamma$ is {\em $\varphi$-saturated} in $\logi$ if $\Gamma \nVdash \varphi$, but $\Gamma,\alpha \Vdash\varphi$ for every $\alpha$ such that $\alpha \notin \Gamma$. It is easy to prove that any $\varphi$-saturated theory is closed. The following useful result will be used along this paper. A proof can be found, for instance, in~\cite[Theorem 22.2]{wojcicki1984lectures}):

\begin{teo}[Lindenbaum-\L os] \label{L-A-lemma} 
Let $\logi$ be a finitary logic, and let $\Gamma\cup \{\varphi\}$ be a set of formulas such that $\Gamma\nVdash \varphi$.  Then, there exists a set of formulas $\Delta$  such that  $\Delta$ is  $\varphi$-saturated in $\logi$ and $\Gamma \subseteq \Delta$.
\end{teo}

Let ${\bb P}$ be a propositional signature. An {\em algebra} for ${\bb P}$ is a pair 	$${\bf A} \  =  \ \brac{A,(\cdot)^\mm A}$$ 
where (1) $A$ is a non-empty set; and (2) $(\cdot)^\mm A$ is a function which maps every $n$-ary connective of  ${\bb P}$ into an  $n$-ary operation on $A$, i.e. if $c\in{\bb P}$ is an $n$-ary connective, then  $c^\mm A$ is a function  $c^\mm A: A^n\imp A$.

As usual, when there is no place for confusion, we shall use the same symbol to design a connective  $c$ and its interpretation $c^\mm A$. 

\begin{ej}\label{ejem1}
Consider the propositional signature  $\bb P_2 \ = \ \set{\wedge,\vee,\neg}$. Let ${\bf DM4}$ be the  element De Morgan algebra given by 
	$${\bf DM4} \ = \ \langle\cuatro,\bb P_2\rangle$$
where the operations are defined in Table \ref{alg-demorgan}.
\end{ej}

\begin{table}[H]
	\centering
	\begin{tabular}{|c||c|c|c|c|}
		\hline
		$\wedge$ & 1	& $\bebe$  & $\nene$  & 0  \\
		\hline
		\hline
		1	& 1  & $\bebe$  & $\nene$  & 0 \\
		\hline
		$\bebe$	& $\bebe$ & $\bebe$  & 0  & 0  \\
		\hline
		$\nene$	& $\nene$ & 0  & $\nene$  & 0 \\
		\hline
		0	& 0 & 0 & 0 & 0 \\
		\hline
	\end{tabular}
	\quad
	\begin{tabular}{|c||c|c|c|c|}
		\hline
		$\vee$	& 1  & $\bebe$ & $\nene$ & 0 \\
		\hline
		\hline
		1	& 1 & 1 & 1 & 1 \\
		\hline
		$\bebe$	& 1 & $\bebe$ & 1 & $\bebe$ \\
		\hline
		$\nene$	& 1 & 1 & $\nene$ & $\nene$ \\
		\hline
		0	& 1 & $\bebe$ & $\nene$ & 0 \\
		\hline
	\end{tabular}
	\quad 
	\begin{tabular}{|c||c|}
		\hline
		&  $\neg$ \\
		\hline
		\hline
		1	& 0 \\
		\hline
		$\bebe$	& $\bebe$ \\
		\hline
		$\nene$	& $\nene$ \\
		\hline
		0	& 1 \\
		\hline
	\end{tabular}
	\caption{Operations of ${\bf DM4}$}
          \label{alg-demorgan}
\end{table}

A {\em logic matrix} (or simply a {\em matrix}) over the propositional signature ${\bb P}$ is a triple $\mm M=\brac{{\mm V},\mm D,\  (\cdot)^\mm M}$ where  $\brac{{\mm V},(\cdot)^\mm M}$ is an algebra for ${\bb P}$ and $\mm D\subseteq {\mm V}$. 
The domain $\mm V$ of the algebra is called {\em the set of truth values} and the elements of $\mm D$ are the {\em designated values} of the matrix.
When there is no doubt about the signature we are working with, the logical matrix will be simply denoted by $\mm M=\brac{{\mm V}, \mm D}$.

Let $\mm M=\brac{{\mm V},\mm D}$ be a matrix over ${\bb P}$. An {\em $\mm M$-valuation} (or  {\em $\mm M$-morphism}) for ${\bb P}$ is a map
$v:\for({\bb P})\imp{\mm V}$ such that for every $n$-ary connective $*\in{\bb P}$  it holds:
	\begin{center}
		$v\left(*(\psi_1,\ldots,\psi_n)\right)=*\left(v(\psi_1),\ldots,v(\psi_n)\right)$.
	\end{center}	
We denote by $\val{\mm M}$ the set of all $\mm M$-valuations. Let $\psi\in\for({\bb P})$, $\Gamma\subseteq\for(\bb P)$ and $v$ an $\mm M$-valuation. We say that $v$ is an {\em $\mm M$-model} of $\psi$ (or that  $v$ \ $\mm M$-satisfies \ $\psi$) if $v(\psi)\in D$. We denote by ${\rm mod}_\mm M(\psi)$ the set of all $\mm M$-models of $\psi$; and we say that $\psi$ is {\em $\mm M$-satisfiable } if ${\rm mod}_\mm M(\psi)\neq\emptyset$.  Besides, $v$ is an {\em $\mm M$-model} of $\Gamma$ if $v\in{\rm mod}_\mm M(\alpha)$ for every $\alpha\in\Gamma$. 
We denote by ${\rm mod}_\mm M(\Gamma):=\bigcap_{\psi\in \Gamma} {\rm mod}_{\mm M}(\psi)$  the set of $\mm M$-models of $\Gamma$; and $\Gamma$ is {\em $\mm M$-satisfiable} if ${\rm mod}_\mm M(\Gamma)\neq\emptyset$.

\

Every matrix  $\mm M$ over a signature $\bb P$ induces a logic $\logi=\brac{\for(\bb P),\models_\logi}$ where
	\begin{center}
		$\Gamma\models_\logi \psi$ \quad iff \quad
		${\rm mod}_\mm M(\Gamma)\subseteq{\rm mod}_\mm M(\psi)$.
	\end{center}
If $\models_\logi\psi$, i.\! e. $v(\psi)\in\mm D$ for every $\mm M$-valuation $v$, we say that $\psi$ is {\em $\logi$-valid} or that it is a  {\em $\logi$-tautology}.
If $\logi$ is the logic induced by the matrix $\mm M$, eventually we shall call $\logi$-valuations to the  $\mm M$-valuations.
The following is a well--known fact.

\begin{teo}\label{carac-mat-fin} \footnote{Cf. 
		\cite[Theorem 3.2.2]{got}, 
		\cite[Proposition 1]{avronipl}, Cf. \cite{shoes71}.
	} 
	Let 	$\mm M=\brac{{\mm V},\mm D,\ {\bb P}}$ be a finite matrix over ${\bb P}$.  Then, the logic 
	$\logi =\brac{\for({\bb P}), \models_\mm M}$
	induced by $\mm M$ is Tarskian and finitary.
\end{teo}

We say that a matrix $\cal M$ is {\em standard} if it has the matrix which induces the positive fragment of classical propositional logic  as a sub-matrix.

 \begin{defi}\footnote{Cf. \cite[Def. 2.7]{coni-es-gi-go}} \label{logi-maxi}
Let $\logi_1$ and $\logi_2$ be two standard logics  (Tarskian and structural) defined over the same propositional signature $\bb P$ such that $\martillo{\logi_1}\ \subsetneq \ \martillo{\logi_2}$. Then $\logi_1$ is said to be  {\em maximal} w.r.t. $\logi_2$ if, for every $\varphi\in\for(\bb P)$ such that  $\martillo{\logi_2}\varphi$ \ but \ $\not \martillo{\logi_1}\varphi$, it holds that the logic $\logi_1^+$, obtained from $\logi_1$ adding $\varphi$ as a theorem, coincides with $\logi_2$
\end{defi}

Let $\logi_1$ and $\logi_2$ be two standard logics  (Tarskian and structural) defined over the same propositional signature $\bb P$ such that $\martillo{\logi_1}\ \subsetneq \ \martillo{\logi_2}$ and let $\circledR$ be a (primitive or defined) connective.  It is called a {\em Derivability Adjustment Theorem} (DAT) to any of the following:
\begin{center}
	$\Gamma\martillo{\logi_2} \psi$ \quad iff \quad
	there exists $\Lambda\subseteq \for(\bb P)$ such that
	$\Gamma,\Lambda^\circledR\martillo{\logi_1}\psi$
 where $\Lambda^\circledR=\{\circledR \gamma :  \gamma\in \Lambda\}$.
\end{center}

\begin{quote}
	$\Gamma\martillo{\logi_2} \psi$ \quad iff \quad
	$\Gamma,\set{\circledR p_1,\dotsc,\circledR p_m}\martillo{\logi_1}\psi$, where $p_1,\dotsc,p_m$ are the propositional letters occurring  in $\Gamma\cup\set\psi$.
\end{quote}

\begin{lema}\footnote{Cf. \cite[Lemma 2.3]{coni-es-gi-go}}\label{submat-sublogica}
	Let $\logi_1$ and $\logi_2$ be the logics induced by the matrices $\brac{{\bf A_1},\mm D_1}$ and $\brac{{\bf A_2},\mm D_2}$, respectively, defined over the signature $\bb P$. If $\bf A_2$ is a subalgebra of $\bf A_1$ and $\mm D_2=\mm D_1\cap A_2$, then \ $\martillo{\logi_1} \ \subseteq \ \martillo{\logi_2}$.
\end{lema}

\begin{teo}\footnote{Cf. \cite[Th. 2.4]{coni-es-gi-go}}\label{logica-maxi-caract}
	Let $\logi_1$ and $\logi_2$ be the logics induced by the matrices $\brac{{\bf A_1},\mm D_1}$ y $\brac{{\bf A_2},\mm D_2}$, respectively, defined over the same propositional signature $\bb P$, such that $\bf A_2$ is a subalgebra of $\bf A_1$ and $\mm D_2=\mm D_1\cap A_2$. Suppose that the following conditions hold:
	\begin{enumerate}[(1).]
		\item $A_1=\set{0,1,a_1,\dotsc,a_k,a_{k+1},\dotsc,a_n}$ and $A_2=\set{0,1,a_1,\dotsc,a_k}$ are finite, $0\not\in \mm D_1$, $1\in\mm D_2$ and $\set{0,1}$ is a subalgebra of $A_2$.
		\item there are $\top(p)$, $\bot(p)\in \for(\bb P)$ such that $e(\top(p))=1$ and $e(\bot(p))=0$ for all  $\logi_1$-valuation $e$.
		\item 
		There exists $\alpha_j^i(p)\in\for(\bb P)$ such that 
		\begin{center}
			$e(p)=a_{ i}$
			\ only if \ 
			$e(\alpha_{j}^{ i}(p))=a_{j}$, 
		\end{center}
		for $k+1\leq i\leq n$, $1\leq j\leq n$, and $i\neq j$.
	\end{enumerate}
	Then,  $\logi_1$ is maximal w.r.t. $\logi_2$.
\end{teo}

\subsection{Logic of formal inconsistency and undeterminedness}

\begin{defi}\label{op-nega}
	Let $\logi=\brac{\for,\vdash_\logi}$ be a propositional Tarskian logic.
\begin{enumerate}[(a)]
		\item a unary connective $\rceil$ (primitive or defined) in the formal language $\rr L$ is a {\em negation} if there is a formula $\psi$ in the language such that
		\begin{center}
			$\psi\not\disdash\logi\negage\psi$ \quad and \quad
			$\negage\psi\not\disdash\logi\psi$
		\end{center}
		\item a binary connective  $\vee$  (primitive or defined) in the formal language $\rr L$ is a {\em  disjunction} if it  holds
\begin{center}
	$\Delta,\, \varphi\disdash\logi \chi$ \ and \ $\Delta,\, \psi\disdash\logi\chi$ \quad iff\quad 
	$\Delta,\varphi\vee\psi\disdash\logi\chi$
\end{center}
	\end{enumerate}
\end{defi}

\
	
Let  $\logi=\brac{\rr L,\vdash_\logi}$ be a Tarskian logic  with  negation connective $\rceil$ (either primitive or defined) in the formal language $\rr L$. We say that $\rceil$ is a {\em classical negation} if it holds.
	\begin{enumerate}[(i)]
		\item $\negage$ is a {\em classical negation} if $\logi$ has a {\em disjunction} $\vee$  (either primitive or defined) and it holds 
	\begin{center}
		\begin{tabular}{llcl}
		{\bf (tnd)} & $\disdash\logi\alpha\vee\ \negage\alpha$ &\hspace{2cm}& {\em tertium non datur} \\ 
		{\bf (ecq)} & $\alpha,\, \negage\alpha\disdash\logi
		\beta$ &
		\hspace{2cm}& {\em ex contradictiorio quodlibet} 
	\end{tabular}
	\end{center}
	\item $\negage$ is a  {\em paraconsistent negation} if it does not satisfies the principle (ECQ), i.e.
	$$\alpha,\negage\alpha\not \disdash\logi \beta$$
		\item $\negage$ is a {\em paracomplete negation} 
	if $\logi$ has a {\em disjunction} $\vee$ (either primitive or defined)  and it does not satisfies the principle (TND), i.e.
	$$\not \disdash\logi \beta\vee\no\beta$$
	\item $\negage$ is a {\em paranormal negation} 
	if $\logi$ has a {\em disjunction} $\vee$ (either primitive or defined)  and $\negage$ is paraconsistent and paracomplete.
	\end{enumerate}


\

According to what  Arieli et al. stated  in \cite[Theorem 1]{paradefinite-avron}, they are necessarily at least four truth values for a logic to enjoy the properties of paraconsistency and paracompleteness, simultaneously.
The paradigm of these four-valued logics, denoted $\bedun$, that was initially developed by  Dunn (cf. \cite{dunn}, \cite{dunn76}) as the semantic for the {\em first degree entailments} logics, being this applied in the realm of Computer Science by Belnap
in his prolific work (cf. \cite{Belnap1}, \cite{Belnap2}).

\begin{defi}\label{belnap-dunn}
	The logic $\bedun$ is the pair $\langle\for(\bb P_2),\models_\bedun\rangle$ where consequence operator $\models_\bedun$ is determined by the matrix $\mm M_\bedun \ = \ \langle {\bf DM4},\mm D_\bedun\rangle$,  ${\bf DM4}$ is the algebra of Example \ref{ejem1} and the set of designated values is $\mm D_\bedun = \set{1,\bebe}$.	
\end{defi}

\begin{teo} The primitive connective $\neg$ of $\bedun$ is a paranormal negation.
\end{teo}
\begin{proof} Let $p,q$ two propositional letters and let $e$ be a $\bedun$-valuation such that  $v(p)=\bebe$ and  $v(q)=\nene$, then  $v\in{\sf mod}_\bedun(p\wedge\no p)$. However,  $v\not\in{\sf mod}_\bedun(q)$; and therefore $\neg$ does not satisfy {\bf (ecq)}. On the other hand, if we choose a propositional letter $p$ and a valuation $v$ such that $v(p)=\nene$, it is clear that $v\not\in {\sf mod}_\bedun(p\vee \no p)$ and hence $\no$ rejects {\bf (tnd)}.
\end{proof}





\begin{defi}\footnote{Cf. \cite{car-con-ro-20}}\label{lfiu}
Let  $\logi$ be a finitary Tarskian logic with a  disjunction connective $\vee$ and a paranormal negation connective $\rceil$. We say that $\logi$ is a  {\em Logic of Formal Inconsistency and Undeterminedness} 
	({\bf LFIU}) for short) if there are unary connectives \ $\circ$ \ and \ $\star$  \ (either primitive or defined) such that
	\begin{multicols}{2}
		\begin{enumerate}
			\item[(a)] $\alpha,\circ\alpha\nturn{\martillo\logi}\beta$, for some $\alpha, \beta$,
			\item[(b)] $\rceil\alpha,\circ\alpha\nturn{\martillo\logi}\beta$, for some $\alpha, \beta$,
			\item[(c)] $\alpha,\rceil\alpha,\circ\alpha\martillo\logi\beta$, for every $\alpha, \beta$,
		\end{enumerate}
		\begin{enumerate}
			\item[(d)] $\nturn{\martillo\logi}\alpha\vee\star\alpha$, for some $\alpha, \beta$,
			\item[(e)] $\nturn{\martillo\logi}\rceil\alpha\vee\star\alpha$, for some $\alpha, \beta$,
			\item[(f)] $\:{\martillo\logi}\:\alpha\vee\rceil\alpha\vee\star\alpha$, for every $\alpha$.
		\end{enumerate}
	\end{multicols}
\end{defi}


On the other hand, in \cite{bscuatro},  Omori and Waragai introduced the four-valued logic $\bede$ pursuing to modify the three-valued logic $\elefi$ (see Section~\ref{sectBB2})  in such a way that it can deal information that not only is inconsistent  but also incomplete. In terms of matrix, this modification corresponds to a move from three-valued matrix to four valued-matrix. In this way, they introduced a four-valued matrix whose induced logic is an extension of $\bedun$ but in the language of $\elefi$ where the operation $\circ$ is defined in such a way that the matrix of $\elefi$, $\mm M_\elefi$, is a sub-matrix of the new one.

\

Let $\bb P_3$ be the propositional signature $\bb P_3=\{\wedge, \vee,  \imp, \neg, \circ, \sim \}$. Now, consider the algebra ${\bf A}_\bede \ = \ \langle\cuatro,\bb P_3\rangle$ whose reduct  $\langle\cuatro,\set{\wedge,\vee,\neg}\rangle$ is the algebra {\bf DM4} and the operations $\imp$, $\sim$ and $\circ$ are given in the Table \ref{alg-bede}; and let 
$\mm M_\bede \ = \ \langle {\bf A}_\bede,\mm D_\bede\rangle$ be the matrix whose set of designated values is $D_\bede=\{1, \bebe\}$. 

\begin{table}[ht]
	\centering
	\begin{tabular}{|c|c|c|c|c|}
		\hline
		$\imp$ & 1 & $\bebe$ & $\nene$ & 0 \\
		\hline
		\hline
		1 & 1 & $\bebe$ & $\nene$ & 0 \\
		\hline
		$\bebe$ & 1 & $\bebe$ & $\nene$ & 0 \\
		\hline
		$\nene$ & 1 & 1 & 1 & 1 \\
		\hline
		0 & 1 & 1 & 1 & 1 \\
		\hline
	\end{tabular}
\quad
\begin{tabular}{|c|c|c|}
	\hline
	& $\circ$ & $\sim$ \\
	\hline
	\hline
	1 & 1 & 0 \\
	\hline
	$\bebe$ & 0 & 0 \\
	\hline
	$\nene$ & 0 & 1 \\
	\hline
	0 & 1 & 1 \\
	\hline
\end{tabular}	
	\caption{Operations of ${\bf A}_\bede$}\label{alg-bede}
\end{table}

Then
\begin{defi}\label{BS4}\footnote{Cf. \cite[Definition 5]{bscuatro}}
$\bede \ = \ \langle\for(\bb P_3),\models_\bedun\rangle$ is the logic induced by the matrix $\mm M_\bede$.
\end{defi}

The simplicity and naturalness of $\bede$ is evidenced by the fact that this same logic was discovered almost simultaneously by different authors from different points of view and diverse contexts. For instance, in his dissertation (see \cite{kleison}) da Silva presents the logic ${\bf BD_\sim}$ that turns out to be equivalent to $\bede$, but his motivations are in the context of logic programming.
On the other hand, in \cite{DeOm} De and Omori  present $\bede$ under the name of ${\sf BDe}$ as one possible extension of $\bedun$ by means of some kind of a classical negation. Other systems equivalent to $\bede$ that can be found in the literature are: \ 
${\bf B_4^{\rightarrow}}$ (cf. \cite{odi}), \ 
${\bf BD}\triangle$ (cf. \cite{sanobde}) \  and \ 
${\sf E_{\sf fde}^\circledast}$ (cf. \cite{ciuni-carrara-18} ).

\

Let $\bullet \alpha := \neg{\circ} \alpha$. Then:

\begin{teo} $\bede$ is a {\bf LFIU} where $\circ$ and $\bullet$ are the respective operators of inconsistency and undeterminedness.
\end{teo}
\begin{proof} Since the underlying algebra of the matrix of $\bedun$ is a reduct of the underlying algebra of the matrix of $\bede$ we have that the negation $\negage$ is paranormal in  $\bede$.
Let $p$ and $q$ be two propositional letters and let $e_1$ and $e_2$ two $\bede$-valuations such that $e_1(p)=1$, $e_1(q)=0$, $e_2(p)=0$ and $e_2(q)=0$ we can see that 
	${\rm mod}_\bede(\set{p,\circ p})\not\subseteq {\rm mod}_\bede(q)$,
	${\rm mod}_\bede(\set{\no p,\circ p})\not\subseteq {\rm mod}_\bede(q)$,
	$e_2\not\in {\rm mod}_\bede(\set{p\vee\bullet p})$, 
	$e_1\not\in {\rm mod}_\bede(\set{\no p\vee\bullet p})$. On the other hand ${\rm mod}_\bede(\set{p,\no p,\circ p})\break =\emptyset$ and ${\rm mod}_\bede(p\vee\no p\vee\bullet p) = \val{\mm M_\bede}$. Therefore, all conditions of Definition~\ref{lfiu} are fulfilled.
\end{proof}

\begin{teo} \label{teoLETs} In $\bede$, the operator $\circ$ recovers {\bf (ecq)} and {\bf (tnd)}.
\end{teo}
\begin{proof} In every {\bf LFI} that extends the classical logic it holds that the consistency operator recovers  {\bf (ecq)}. Suppose that  $v(\circ \alpha)=1$, then $v(\alpha)\in\set{0,1}$. Therefore,  $\circ\alpha\dismodels\bede\alpha\vee\no\alpha$.
\end{proof}

\begin{nota} \label{obsLETs}
In a series of papers, Carnielli and Rodrigues introduced a family of logics designed to deal with paraconsistency and  paracompleteness from an epistemic and informational perspective (see, for instance, \cite{car:rod:17}). These logics are called {\em em Logics of Evidence and Truth} ({\bf LET}s for short). {\bf LET}s are expansions of  $\bedun$ by adding a (primitive or not) {\em classicality operator} $\circ$ which simultaneously   recovers  {\bf (ecq)} and {\bf (tnd)}. In that sense, the logic  $\bede$ is a  {\bf LET}. As we shall see, $\bededos$ to be considered in the following sections is also a {\bf LET}.
\end{nota}

\section{The sentential logic $\bededos$} \label{sectBB2}

The logic $\bededos$ was introduced in \cite{DeOm} as an expansion of $\bedun$ by means of a classical negation which satisfies certain requirements. However, in our research, we were looking for an extension $\bedun$ (in the context of the LFIs) by means of a consistency operator $\copi$ that might take values different from 0 and 1. In this way, we arrive to a systems that is equivalent to $\bededos$ following a different path. This is due to the fact that the classical negation proposed by  Omori et al. and our consistency operator are inter-definable. Consequently, we can assert that $\bededos$ is a LFU with an underterminedness operator definable from a consistency operator.

\

Let us consider the propositional signature $\bb P_T=\set{\vee, \wedge,\imp,\neg,\copi}$.
\begin{defi}\label{BD2}\footnote{Cf. \cite{DeOm}. It was originally defined over a different signature.}
The logic $\bededos \ = \ \langle\for(\bb P_T),\models_\bededos\rangle$ is the logic induced by the matrix $\mm M_\bededos \ = \ 
\langle {\bf A}_\bededos,\mm D_\bededos\rangle$
where  ${\bf A}_\bededos \ = \ \langle\cuatro,\bb P_T\rangle$ is the algebra whose reduct
$\langle\cuatro,\set{\wedge,\vee,\neg}\rangle$ is  {\bf DM4}, 
the operations $\imp$ and $\copi$ are given in the Table \ref{alg-bededos} and the set of designated values is $\mm D_\bededos = \set{1,\bebe}$.

\begin{table}[h]
	\centering
	\begin{tabular}{|c|c|c|c|c|}
		\hline
		$\imp$ & 1 & $\bebe$ & $\nene$ & 0\\
		\hline
		\hline
		1  & 1 & $\bebe$ & $\nene$ & 0 \\
		\hline
		$\bebe$  & 1 & $\bebe$ & $\nene$ & 0 \\
		\hline
		$\nene$  & 1 & $\bebe$ & 1 & $\bebe$ \\
		\hline
		0  & 1 & 1 & 1 & 1 \\
		\hline
	\end{tabular}
\quad
\begin{tabular}{|c|c|}
	\hline
	& $\copi$ \\
	\hline
	\hline
	1&  1\\
	\hline
	$\bebe$&0\\
	\hline
	$\nene$&$\bebe$ \\
	\hline
	0&1 \\
	\hline
\end{tabular}
	\caption{Implication and consistency operator of ${\bf A}_\bededos$}\label{alg-bededos}
\end{table}
\end{defi}

Consider now the operators \ $\nof$, \ $\circ$, \ $\estar$ and \ $\bestar$:

\begin{center}

\begin{tabular}{|c|c|c|c|c|}
	\hline
	& $\sim$ & $\circ$ & $\estar$ & $\bestar$ \\
	\hline
	\hline
	1& 0 & 1 & 0 & 1\\
	\hline
	$\bebe$&0&0&0 &1\\
	\hline
	$\nene$&$\bebe$&0&1&0 \\
	\hline
	0&1&1&0&1 \\
	\hline
\end{tabular}
\end{center}

It is immediate to see that they  are definable  in $\bededos$ as follows: 

\begin{align*}
\nof x &\ := \ \copi x\wedge\neg(x\wedge\copi x), & 
\bestar x &\ := \ \copi\copi x, \\ 
\estar x &\ := \ \neg\,\copi\copi x,&
\circ x &\ := \ \copi x\wedge\bestar x.
\end{align*}

Moreover, $\vee$ and $\to$ can be defined as follows:\\

$x\vee y \ := \ \no(\no x\wedge\no y)$ \ and \ $x\rightarrow y \ := \ \no \big( \no (\copi x\wedge\no(x\wedge\copi x))
\ \wedge \ \no y \big)$.

\begin{propo}\label{bededos-stand}
	$\bededos$ is sound w.r.t the positive propositional logic, i.e. in $\bededos$ the following axiom schemata are valid
\begin{center}
			\begin{tabular}{llll}
		\pos 1 & $\alpha\imp(\beta\imp\alpha)$&
		\pos 7 & $\alpha\imp(\beta\vee\alpha)$\\[2mm]
		\pos 2 & $(\alpha\imp(\beta\imp\gamma))\imp$ &
		\pos 8 & $(\alpha\imp\gamma)\imp((\beta\imp\gamma)\imp$\\[2mm]
		&$((\alpha\imp\beta)\imp(\alpha\imp\gamma))$&
		&$(\alpha\vee\beta\imp\gamma))$ \\[2mm]
		\pos 3 & $\alpha\wedge\beta\imp\alpha$&
		\pos 9 & $\alpha\vee (\alpha\imp\beta)$\\[2mm]
		\pos 4 & $\alpha\wedge\beta\imp\beta$\\[2mm]
		\pos 5 & $\alpha\imp(\beta\imp(\alpha\wedge\beta))$\\[2mm]
		\pos 6 & $\alpha\imp(\alpha\vee\beta)$\\[2mm]
	\end{tabular} 
\end{center}	
\end{propo}
\begin{proof} This follows from the fact that the sub-matrix
	$\langle\cuatro,\mm D_\bededos,\set{\vee,\wedge,\imp}\rangle$ of
	$\mm M_\bededos$ is a standard matrix (cf. \cite[p. 25-26]{rostur},\, 	
	\cite[p. 326]{malino-many},\,
	\cite[p. 30]{got},\, 
	\cite[Definition 10]{paradefinite-avron}
	).	
\end{proof}

\begin{propo}\label{negclasica}
The connective $\sim$ is a classical negation for $\bededos$.
\end{propo}
\begin{proof}
It is clear, from the definition of $\nof$ (see Table \ref{alg-bededos}) that for all $\alpha\in\for(\bb P)$, it holds
${\rm mod}_\bededos(\set{\alpha,\nof\alpha})=\emptyset$ and ${\rm mod}_\bededos(\alpha\vee\nof\alpha)=\val{\mm M_\bededos}$.
Therefore,  $\nof$ satisfies {\bf (ecq)} and {\bf (tnd)}.
\end{proof}

\begin{nota} In $\bededos$, the negation $\nof$ is a classical one in the sense that it validates {\em ex contradictione quodlibet} \, $\alpha,\nof\alpha\mobedos\beta$ \,
	and {\em tertium non datur} \, $\mobedos\alpha\vee\nof\alpha$.
\end{nota}

\begin{propo}\label{bededos-LFIU}
The logic $\bededos$ is a {\bf  LFIU}. Moreover,  $\copi$ \ is the consistency operator and \ $\estar$ \ is the operator of undeterminedness. Besides, the operator $\bestar$ recovers {\em tertium non datur}.
\end{propo}
\begin{proof}
The negation $\no$ is paracomplete since it coincides with the negation of $\bedun$. Let $p$ and $q$ two propositional letters and consider the $\bededos$-valuation $v$ such that $v(p)=1$ and $v(q)=0$. Then, it is clear that ${\rm mod}_\bededos(\set{\copi p,p})\not\subseteq {\rm mod}_\bededos(q)$ and $v\not\in{\rm mod}_\bededos(\no p\vee\estar p)$ (or equivalently, ${\rm mod}_\bededos(\bestar p)\not\subseteq {\rm mod}_\bededos(\no p)$). On the other hand, considering the $\bededos$-valuation $w$ such that $w(p)=0$ and $w(q)=1$ we can see that  
${\rm mod}_\bededos(\set{\copi p,\no p})\not\subseteq {\rm mod}_\bededos(q)$ and
$v\not\in{\rm mod}_\bededos(p\vee\estar p)$ (or equivalently, ${\rm mod}_\bededos(\bestar p)\not\subseteq {\rm mod}_\bededos(p)$).

Finally, since for all $\alpha\in\for$ we have ${\rm mod}_\bededos(\set{\alpha,\no\alpha,\copi\no\alpha})=\emptyset$ and ${\rm mod}_\bededos(\alpha\vee\no\alpha\vee\estar\alpha)=\val{\mm M_\bededos}$ (or ${\rm mod}_\bededos(\bestar\alpha)\subseteq{\rm mod}_\bededos(\alpha\vee\no\alpha)$), the assertion holds.  
\end{proof}

\begin{propo}
$\bededos$ is finitary.
\end{propo}
\begin{proof}
From Theorem \ref{carac-mat-fin}.
\end{proof}

\begin{nota} By Theorem~\ref{teoLETs} adapted to $\bededos$, the operator $\circ$ is a classicality operator, in the sense that it recovers {\bf (ecq)} and {\bf (tnd)}. This means that $\bededos$ is a {\bf LET}, recall Remark~\ref{obsLETs}. The relationship of   $\bededos$ with other {\bf LET}s studied in the literature is a subject that deserves future research.
\end{nota}

The three-valued paraconsistent logic $\elefi$, introduced in~\cite{car:mar:dea:00}, was independently proposed by various authors. Among othters, it is equivalent (up to signature) to the well-known paraconsistent logic {\bf J3} Introduced by D'Ottaviano and da Costa in~\cite{dot:dac:70} (for more details about this topic consult, for instance, Sections~4.4.3 and~4.4.7 of~\cite{CarCon}). 

Recall Definition~\ref{BD2}.  Observe that $\{1, \bebe,0\}$ is a subalgebra of ${\bf A}_\bededos$ in which $\copi{x}=\circ x$, for every $x$.  As shown in~\cite[Section~4.4.7]{CarCon}, $\elefi$ can be defined over signature $\{\land, \vee, \to, \neg, \circ \}$ or, equivalently (by the last observation), over signature $\bb P_T$. The corresponding matrix logic with domain $\{1, \bebe,0\}$ and set of designated values $D_\bededos=\{1, \bebe\}$ is therefore a presentation of  $\elefi$ over  $\bb P_T$ as a sub-matrix of  $\bededos$.
The next proposition is a first step in order to prove that $\bededos$ is maximal with respect $\elefi$ by using Theorem~\ref{logica-maxi-caract}).

\begin{propo}\label{extensiones-bd2}
 $\cele$  and $\elefi $ are deductive extensions of $\bededos$. 
\end{propo}
\begin{proof}
Indeed, $\brac{\set{1,0},\bb P_T}$ is a sub-algebra of ${\bf A}_\bededos \ = \ \langle\cuatro,\bb P_T\rangle$. Besides,  $\mm D_\cele=\set 1=\mm D_\bededos \cap \{0,1\}$. By Lemma~\ref{submat-sublogica}, we have that $\cele$  is a deductive extension of $\bededos$. The proof that $\elefi$ is a deductive extension of $\bededos$ is analogous.
\end{proof}

\begin{teo}\label{maximajotatres}
	$\bededos$ is maximal w.r.t.  $\elefi$.
\end{teo}
\begin{proof} Let us see that the conditions of Theorem \ref{logica-maxi-caract} are verified. In the proof of Proposition \ref{extensiones-bd2}, we checked that the basic conditions for each respective matrix hold. Besides, ``(1)'' ${\bf A}_\bededos=\set{1,0,\bebe,\nene}$ and ${\bf A}_\elefi=\set{1,0,\bebe}$ are finite, $0\not\in\mm D_{\bededos}$, $1\in\mm D_\elefi$ and $\set{0,1}$ is a subalgebra of ${\bf A}_\elefi$; \ \ ``(2)'' Let $\top(p):=\estar p\vee\bestar p$ and $\bot(p):=p\wedge\no p\wedge\copi p$, then $v(\top(p))=1$ and $v(\bot(p))=0$ for all $\bededos$-valuation $v$; \ \ ``(3)'' Let us consider the formula $\varphi_\bebe^\nene(p)=\copi p$. Then, for every $\bededos$-valuation $v$, $e(p)=\nene$ implies that $e\left(\varphi_\bebe^\nene(p)\right)=\bebe$.	 Therefore, $\bededos$ is maximal w.r.t. $\elefi$ by  Theorem \ref{logica-maxi-caract}.
\end{proof}

\

We end this section exhibiting some {\em DAT}'s.

\begin{teo}\label{dat1}	If $\Gamma$ is a finite set of formulas, then 
	\begin{center}
			$\Gamma\dismodels\elefi\psi$ \ iff \ $\Gamma,\bestar p_1,\dotsc,\bestar p_k\mobedos\psi$
	\end{center}
	where $p_1,\dotsc,p_k$ are the propositional variables occurring in $\Gamma\cup\set\psi$.
\end{teo}
\begin{proof}{ \em ``Only if''} \ part. Let  $v$ a $\mm M_\bededos$-valuation such that $v[\Gamma\cup\set{\bestar p_1,\dotsc,\bestar p_k}]\subseteq\set{1,\bebe}$. Then $v(p_i)\in\set{1,\bebe,0}$ for all  $1\leq i\leq n$. Let $\bar v$ be the  $\mm M_\elefi$-valuation defined  by
	$$\bar v(p) \ := \ 
	\begin{cases}
		v(p)& \mbox{\ if \ } p\in\set{p_1,\dotsc,p_k}\\
		0& \mbox{\ if \ } p\not\in\set{p_1,\dotsc,p_k}\\
	\end{cases}
	$$
	Is clear that $\bar v(\alpha)=v(\alpha)$ for every formula $\alpha$ such that ${\sf var}(\alpha)\in\set{p_1,\dotsc,p_k}$. Then, 
	$\bar v[\Gamma]\subseteq\set{1,\bebe}$ and by hypothesis we have $\bar v(\psi)\in\set{1,\bebe}$. Therefore, $v(\psi)\in\set{1,\bebe}$.
	\\[2mm]
	{\em ``If''} \ part. Let $v$ be a $\mm M_\elefi$-valuation  such that  $v[\Gamma]\subseteq\set{1,\bebe}$. Let us consider the $\mm M_\bededos$-valuation $\bar v$ where $\bar v=v$. Then, $\bar v[\Gamma]\subseteq\set{1,\bebe}$ and $\bar v(\bestar p_i)=1$ for each $1\leq i\leq k$. By hypothesis, it follows that $v(\psi)=\bar v(\psi)\in\set{1,\bebe}$.
\end{proof}

\

\begin{teo} \label{dat2} Consider the logic $\cele$ presented over the signature $\bb P_T$ with matrix $\mm M_\cele=\brac{{\bf B},\mm D_\cele}$, where ${\bf B}=\brac{\set{0,1},\bb P_T}$ is the subalgebra of ${\bf A_\bededos}$. 

		If $\Gamma$ is a finite set of formulas, then
		\begin{center}
			$\Gamma\dismodels\cele\psi$ \ iff \ $\Gamma,\circ p_1,\dotsc,\circ p_k\mobedos\psi$
		\end{center}
		where $p_1,\dotsc,p_k$ are the propositional variables occurring in $\Gamma\cup\set\psi$.
\end{teo}
\begin{proof} It is similar to the proof of Theorem~\ref{dat1}.
\end{proof}

Obviously Theorem \ref{dat1} can be generalized to any (finite or non-finite) set of formulas:

\begin{teo} Let $\Gamma \cup\{\psi\}$ be a set of formulas. Then, 
		\begin{center}
			$\Gamma\dismodels\elefi\psi$ \ iff \ $\Gamma,\left({\sf var}(\Gamma)\right)^\bestar,\left({\sf var}(\psi)\right)^\bestar\mobedos\psi$
		\end{center}
		where $\left({\sf var}(\Gamma)\right)^\bestar=\set{\bestar p\mid p\in{\sf var}(\Gamma)}$ \ and \
		$\left({\sf var}(\psi)\right)^\bestar=\set{\bestar p\mid p\in{\sf var}(\psi)}$.
\end{teo}

\

The same can be done with Theorem \ref{dat2}.

\section{Twist structures and a Hilbert calculus for $\bededos$} \label{twist}

In this section a Hilbert style presentation for $\bededos$ will be given, obtained from a semantical characterization by means of twist structures.

Twist structures were introduced independently by M. Fidel (\cite{fid:78}) and D. Vakarelov (\cite{vaka:77}) in 1977-78 with the aim of characterizing Nelson's logic {\bf N4}. However, the main algebraic ideas underlying twist structures were already introduced in 1958 by J. Kalman (\cite{kal:58}). Indeed, in such a work Kalman introduced the now called De Morgan lattices starting from a distributive lattice $L$ in which he defined the following operations over $L \times L$:

\begin{itemize}
\item[(1)]   $\tilde{\neg}\,(a,b) = (b, a)$;
\item[(2)]  $(a,b) \,\tilde{\wedge}\, (c,d) = (a \sqcap c, b \sqcup d)$; 
\item[(3)] $(a,b) \,\tilde{\vee}\, (c,d) = (a \sqcup c, b \sqcap d)$ 
\end{itemize}

\noindent  where $\sqcap$ and $\sqcup$ denote infima and suprema in $L$.
This produces a De Morgan lattice. The same idea was proposed by Fidel and Vakarelov for {\bf N4}, and also by J.~M. Dunn in~\cite{dunn}, where he obtained a representation of De Morgan lattices by means of pairs of sets with the same kind of operations. Twist structures semantics were afterwards proposed for several logics. In the sequel,  a twist structures semantics for    $\bededos$ will  be introduced.

First, consider for convenience the  signature $\mathbb{P}_T=\{\land, \vee, \to, \neg,\copi\}$ for   $\bededos$. The presentation of $\bededos$ as a matrix logic over signature $\mathbb{P}_T$ will be denoted by $\bededos_T$. The idea is, as in the case of $\bedun$, identifying the truth values of {\bf 4} with pairs $(a,b)$ in ${\bf 2}=\{0,1\}$ such that $a$ represents (in informal terms) information for a given sentence $\alpha$ while $b$ represents information about the negation of $\alpha$ (or against $\alpha$). Hence, $1$, $\bebe$, $\nene$ and $0$ are identified, respectively with  $(1,0)$, $(1,1)$, $(0,0)$ and $(0,1)$.\footnote{This is analogous to the interpretation of Belnap and Dunn's logic $FDE$.} Consider the following twist operators over  ${\bf 2} \times {\bf2}$: $\tilde{\neg}$, $\tilde{\wedge}$, $\tilde{\vee}$ are defined by means of (1)-(3) above, and

\begin{itemize}
\item[(4)] $(a,b) \,\tilde{\to}\, (c,d)=(a \Rightarrow c, (b \Rightarrow a) \sqcap d)$;
\item[(5)] $\tilde{\copi{}} (a,b)= ({\sim}(a \sqcap b), a \Leftrightarrow b)$
\end{itemize}

\noindent where $a \Rightarrow b = {\sim} a \sqcup b$ and ${\sim} a$ denote respectively the implication and the Boolean complement in {\bf 2} (seen as a Boolean algebra), and $a \Leftrightarrow b=(a \Rightarrow b) \sqcap (b \Rightarrow a)$. Let $\mathcal{T}_{\bf 2}$ be the (twist) algebra obtained in this way, and let $\mathcal{M}_{\bf 2}= \langle \mathcal{T}_{\bf 2},D_{\bf 2}\rangle$ be the logical matrix associated to $\mathcal{T}_{\bf 2}$ where $D_{\bf 2}=\{(1,0), (1,1)\}$. Observe that, with the identifications above, these operations correspond to the truth-tables of the operators in the signature $\mathbb{P}_T$, hence $\mathcal{M}_{\bf 2}$ is (up to names) the logical matrix for $\bededos_T$. 

This construction can be easily defined over any Boolean algebra {\bf A}: indeed, let  $\mathcal{T}_{\bf A} = \langle A \times A, \tilde{\wedge}, \tilde{\vee}, \tilde{\to}, \tilde{\neg}, \tilde{\copi{}} \rangle$ such that the operations are defined by means of clauses (1)-(5) above (where now $a,b,c,d \in A$, the domain of {\bf A}). The logical matrix associated to $\mathcal{T}_{\bf A}$ is $\mathcal{M}_{\bf A}= \langle \mathcal{T}_{\bf A},D_{\bf A}\rangle$  such that $D_{\bf A}=\{(1,a) \ : \ a \in A\}$ and $1$ is the top element of {\bf A}. Let $\sdtstile{\bf A}{T}$ be the consequence relation associated to $\mathcal{M}_{\bf A}$, namely: $\Gamma \sdtstile{\bf A}{T} \varphi$ iff, for every valuation $v$ over $\mathcal{M}_{\bf A}$, if $v(\gamma) \in D_{\bf A}$ for every $\gamma \in \Gamma$ then $v(\varphi) \in D_{\bf A}$. Let  $\sdtstile{\bededos}{T}$ be the consequence relation associated to the class of twist structures for $\bededos_T$, namely: $\Gamma \sdtstile{\bededos}{T} \varphi$ iff $\Gamma \sdtstile{\bf A}{T} \varphi$ for every {\bf A}.

From this semantics, a Hilbert calculus for $\bededos$ (to be precise, for $\bededos_T$) will be presented. It is easy to see  that, over the signature $\{\land,\vee, \to\}$, the logic   $\bededos_T$ coincides with positive  classical logic. Taking this into account, and the definition of the twist operator  $\tilde{\copi{}}$, we obtain the following Hilbert calculus $\hachedost$ for $\bededos_T$ (where $\alpha \leftrightarrow \beta$ is an abbreviation for $(\alpha \to \beta) \land (\beta \to \alpha)$):\\[2mm]

\noindent {\bf Axioms} \\

Axiom schemata \pos 1-\pos 9 from positive classical logic (recall Proposition~\ref{bededos-stand}), plus the following ones:

\

	\begin{tabular}{llr}
		{\bf (DNeg)}
		&$\neg\neg\alpha \leftrightarrow \alpha$&\\[2mm]
		{\bf (DM1)}
		&$\neg(\alpha \vee \beta) \leftrightarrow (\neg \alpha\wedge \neg\beta)$&\\[2mm]
		{\bf (DM2)}
		&$\neg(\alpha \wedge \beta) \leftrightarrow (\neg \alpha\vee \neg\beta)$&\\[2mm]
		{\bf (DM3)}
		&$\neg(\alpha \to \beta) \leftrightarrow ((\neg \alpha \to \alpha)\wedge \neg\beta)$&\\[2mm]
		{\bf ($\copi$-1)}
		&$(\copi{\alpha} \land (\alpha \land \neg \alpha)) \to \beta$&\\[2mm]
		{\bf ($\copi$-2)}
		&$\copi{\alpha} \vee (\alpha \land \neg \alpha)$&\\[2mm]
		{\bf ($\copi$-3)}
		&$\neg \copi{\alpha}  \leftrightarrow (\alpha  \leftrightarrow \neg \alpha)$&\\
	\end{tabular}

\

\

\noindent{\bf Rule of inference} \, (MP) \ modus ponens.

\

In order to prove soundness and completeness of $\hachedost$
w.r.t.  the four-valued matrix logic $\bededos_T$ it will be first proven soundness and completeness of   $\hachedost$  w.r.t. the twist structures semantics defined above. Hence, the former result will be obtained from the latter by using standard arguments from Boolean algebras.\footnote{A similar technique was used in~\cite[Section~5]{bor:con:her:22}.} Before doing this, a useful notation will be introduced. Let $v:For(\mathbb{P}_T) \to\mathcal{T}_{\bf A}$ be a valuation over $\mathcal{M}_{\bf A}$ (that is, a homomorphism of $\mathbb{P}_T$-algebras). From now on, $v$ can be denoted by $v=(v_1,v_2)$ where $v_1,v_2:For(\mathbb{P}_T) \to A$. Hence, $v(\alpha)=(v_1(\alpha),v_2(\alpha))$ for every formula $\alpha$.

\begin{teo} [Soundness and completeness of $\hachedost$ w.r.t. twist structures semantics] \label{sound-compl-twist}
Let $\Gamma \cup \{\alpha\}$ be  a set of formulas over $\mathbb{P}_T$. Then: $\Gamma \mabedost \alpha$ \ iff \ $\Gamma \sdtstile{\bededos}{T} \alpha$.
\end{teo}
\begin{proof} \ \\
{\em (Soundness)}:  Let  {\bf A} be a Boolean algebra. It is immediate to see that, for any instance $\alpha$ of an axiom of $\hachedost$, it is the case that $v_1(\alpha)=1$ for every valuation $v$ over  $\mathcal{M}_{\bf A}$. That is, $v(\alpha) \in D_{\bf A}$ for every axiom $\alpha$. For instance, since $v(\copi{\alpha})=\tilde{\copi{}}(v(\alpha))$ and  $v(\neg \alpha)=\tilde{\neg}(v(\alpha))$ then $v_1(\copi{\alpha})={\sim}(v_1(\alpha) \sqcap v_2(\alpha)) = {\sim}(v_1(\alpha) \sqcap v_1(\neg\alpha))$ (by definition of $\tilde{\copi{}}$ and $\tilde{\neg}$). Hence, $v_1(\alpha)=1$ for every instance $\alpha$ of axioms {\bf ($\copi$-1)} and {\bf ($\copi$-2)}.
In addition, it is clear that $v_1(\alpha \to \beta)=v_1(\alpha)=1$ implies that $v_1(\beta)=1$, given that $v_1(\alpha \to \beta)=v_1(\alpha) \Rightarrow v_1(\beta)$. From this, by induction on the length of a derivation in  $\hachedost$ of $\varphi$ from $\Gamma$ it can be proven the following:  $\Gamma \mabedost \varphi$ implies that $\Gamma \sdtstile{\bf A}{T} \varphi$ for every {\bf A}.\\[1mm]
{\em (Completeness)}: Assume that $\Gamma \not{\mabedost} \varphi$.  By Theorem~\ref{L-A-lemma} there exists a $\varphi$-saturated set $\Delta$ in $\hachedost$ containing $\Gamma$. Consider now the following relation in $For(\mathbb{P}_T)$: $\alpha \equiv_\Delta \beta$ iff $\Delta \mabedost \alpha \leftrightarrow \beta$.  By using the axioms of positive classical logic  it is immediate to see that $\equiv_\Delta$ is a congruence over $For(\mathbb{P}_T)$ w.r.t. the connectives $\land$, $\lor$ and $\to$. That is, $[\alpha] \sqcap [\beta]=[\alpha \land \beta]$,   $[\alpha] \sqcup [\beta]=[\alpha \vee \beta]$ and  $[\alpha] \Rightarrow [\beta]=[\alpha \to \beta]$ are well-defined operations, where $[\gamma]$ denotes the equivalence class of the formula $\gamma$ w.r.t.  $\equiv_\Delta$. Moreover, $A_\Delta := For(\mathbb{P}_T)/_{\equiv_\Delta}$ is the domain of a Boolean algebra ${\bf A}_\Delta$ in which $1=[\alpha \to\alpha]$ and $0=[\copi{\alpha} \land \alpha \land \neg\alpha]$ for any formula $\alpha$. Hence $[\sneg\beta] = \sneg[\beta]$ in the Boolean algebra  ${\bf A}_\Delta$. Define now the twist structure  $\mathcal{T}_{{\bf A}_\Delta}$ over ${\bf A}_\Delta$, as well as its associated logical matrix  $\mathcal{M}_{{\bf A}_\Delta}$. 

By the very definitions, $[\alpha]=1$ iff $\Delta \mabedost \alpha$, iff $\alpha \in \Delta$. Now, it is easy to prove that the function $v_\Delta:For(\mathbb{P}_T) \to  A_\Delta \times A_\Delta$ given by $v_\Delta(\alpha)=([\alpha], [\neg\alpha])$ is  a valuation over $\mathcal{M}_{{\bf A}_\Delta}$. This is a consequence of the axioms of $\hachedost$ and  the definition of the operations in the Boolean algebra ${\bf A}_\Delta$. From this,  $v_\Delta$ is a valuation over  $\mathcal{M}_{{\bf A}_\Delta}$ such that $v_\Delta(\gamma)  \in D_{{\bf A}_\Delta}$ for every $\gamma \in \Gamma$, but $v_\Delta(\varphi)  \notin D_{{\bf A}_\Delta}$, since $\varphi \notin \Delta$. This proves that  $\Gamma \not\sdtstile{\bededos}{T} \varphi$.
\end{proof}

Now, it will be convenient to recall some classical results concerning Boolean algebras. The interested reader can consult, for instance,~\cite{giv:hal:2009}:

\begin{propo} \label{BAP}
Let {\bf A} be a Boolean algebra with domain $A$. Then:\\[1mm]
(1) If $a$ is an element of $A$ different from 1,  there exists an ultrafilter $F$ over {\bf A} such that $a \notin F$.\\[1mm]
(2) If $F$ is an ultrafilter over {\bf A}, the characteristic map $h_F:A \to \{0,1\}$ of $F$, given by $h_F(x)=1$ iff $x \in F$, is a homomorphism of Boolean algebras  between {\bf A} and the two-element Boolean algebra ${\bf 2}$.
\end{propo}

\begin{teo}  [Soundness and completeness of $\hachedost$  w.r.t.  $\bededos_T$] \ \\ \label{adequacy}
Let $\Gamma \cup \{\varphi\}$ be a set of formulas in $For(\mathbb{P}_T)$. Then, $\Gamma \mabedost \varphi$ iff $\Gamma \mobedos_T \varphi$. 
\end{teo}
\begin{proof} \ \\
{\em (Soundness)}: It follows from Theorem~\ref{sound-compl-twist} (Soundness) and from the fact that the four-valued logical matrix  $\bededos_T$  coincides (up to names) with  $\mathcal{M}_{\bf 2}$, the logical matrix associated to the Boolean algebra {\bf 2}.\\[1mm]
{\em (Completeness)}:  Assume that $\Gamma \not{\mabedost} \varphi$. By using completeness of $\hachedost$ w.r.t. twist structures semantics (see Theorem~\ref{sound-compl-twist}), it follows that $\Gamma \not\sdtstile{\bededos}{T} \varphi$. This means that there exists a Boolean algebra {\bf A} and a valuation $v$ over  $\mathcal{M}_{\bf A}$ such that $v(\gamma) \in D_{\bf A}$ for every $\gamma \in \Gamma$, but $v(\varphi) \notin D_{\bf A}$. That is, $v_1(\gamma)=1$ for every $\gamma \in \Gamma$, but $v_1(\varphi) \neq 1$. By Proposition~\ref{BAP} item~(1), there exists an ultrafilter $F$ over {\bf A} such that $v_1(\varphi) \notin F$. Let $h_F:A \to \{0,1\}$ be the characteristic map of $F$. By Proposition~\ref{BAP}  item~(2), $h_F$ is a homomorphism of Boolean algebras between {\bf A} and the two-element Boolean algebra ${\bf 2}$. Define now the function $\bar{v}:For(\mathbb{P}_T) \to {\bf 2} \times {\bf 2}$ given by $\bar{v}(\alpha)=(h_F(v_1(\alpha)),h_F(v_2(\alpha)))$. Taking into account that $h_F$ is a homomorphism of Boolean algebras, that   $v$ is a valuation over  $\mathcal{M}_{\bf A}$, and by the definition of the operations in the twist structure  $\mathcal{T}_{\bf 2}$,  it is easy to prove that $\bar{v}$ is a valuation over the matrix  $\mathcal{M}_{\bf 2}$ (that is, over the four-valued matrix of $\bededos_T$, up to names).
Moreover, the valuation $\bar{v}$ is  such that $\bar{v}(\gamma) \in D_{\bf 2}$ for  every $\gamma \in \Gamma$  (since $v_1(\gamma)=1$, hence $h_F(v_1(\gamma)) =1$), but $\bar{v}(\varphi) \notin D_{\bf 2}$ (since $v_1(\varphi) \notin F$, hence $h_F(v_1(\varphi)) =0$). From this, $\Gamma \not\models_{\mathcal{M}_{\bf 2}} \varphi$. That is, $\Gamma \not\mobedos_T \varphi$. 
\end{proof}

\begin{propo}\label{esq-validos} The following schemes are provable in $\hachedos_T$.
		\begin{itemize}
			\item[(i)] $(\copi\alpha\imp\alpha)\imp\alpha$
			\item[(ii)] $\big((\no\alpha\imp\alpha)\wedge(\no\copi\alpha\imp\alpha)\big)\imp\alpha$
			\item[(iii)] $(\alpha\wedge\copi\alpha\wedge\no\copi\alpha)\imp\beta$
			\item[(iv)] $(\no\copi\alpha\imp\copi\alpha)\imp\copi\alpha$
			\item[(v)] $(\copi\alpha\imp\no\alpha)\imp\no\alpha$
			\item[(vi)]
			$\copi\alpha\sii\copi\no\alpha$
			\item[(vii)] $\no\copi\alpha\sii\no\copi\no\alpha$
			\item[(viii)] $(\alpha\imp\copi\copi\alpha)\wedge(\no\alpha\imp\copi\copi\alpha)$
			\item[(xi)] $(\alpha\wedge\no\alpha)\sii(\no\copi\alpha\wedge\copi\copi\alpha)$
			\item[(x)] $\no\alpha\imp\no(\alpha\wedge\beta)$
			\item[(xi)] $(\no\alpha\wedge\copi\alpha)\imp\copi(\alpha\wedge\beta)$
			\item[(xii)] $(\copi\alpha\wedge\no\copi\alpha)\imp\copi(\alpha\wedge\beta)$
			\item[(xiii)] $(\copi\alpha\wedge\copi\beta)\imp\copi(\alpha\wedge\beta)$
			\item[(xiv)] $\big((\alpha\wedge\copi\alpha)\wedge\no\copi\beta\big)\imp\no\copi(\alpha\wedge\beta)$
			\item[(xv)]
			$\big((\copi\alpha\wedge\no\copi\alpha)\wedge(\copi\beta\wedge\no\copi\beta)\big)\imp\no\copi(\alpha\wedge\beta)$		
		\end{itemize}
	\end{propo}
\begin{proof}
It is easy to prove that each schema (i)-(xv) is a tautology in $\bededos_T$. Hence, they are provable in  $\hachedos_T$, by Theorem~\ref{adequacy}.
\end{proof}

\section{The first-order logic $\cube$}

Consider the language $\rr L=\langle \mm Q,{\bb P}_T,\mm S\rangle$ where $\mm Q=\set{\pt,\ex}$,  $\mathbb{P}_T=\{\land, \vee, \to, \neg,\copi\}$ is the propositional signature of $\bededos_T$ and $\mm S=\mm P\cup\mm F\cup \mm C$ is a first-order signature formed by the (disjoint) sets  $\mm P\neq\emptyset$, $\mm F\neq\emptyset$ and $\mm C$ containing predicate, function and constant symbols, respectively. Besides, let $\mm V$ be a denumerable set of variable symbols. We denote by $\for(\mm S)$ the set of all (well-formed) formulas of $\rr L$, $\sent(\mm S)$ the set of all closed formulas of $\rr L$, $\ter(\mm S)$ the set of all terms of $\rr L$ and ${\sf clo}(\mm S)$ the set of all closed terms of $\rr L$.

Given $\varphi\in\for(\mm S)$ we denote by $\libre \varphi$ the set of all variables that occur free in  $\varphi$ and by $\libre{x,\varphi}$ the set of all terms free for the variable $x$ in $\varphi$. The complexity of a formula (term) is defined as usual.
\

\begin{defi}\label{parcial}
	We call a $\cuatro$-structure (or partial structure) over the first-order signature  $\mm S=\brac{\mm P,\mm F,\mm C}$ to any pair
	$$\ff A=\brac{A,(\cdot)^\ff A}$$
	where $A\neq \emptyset$ and $(\cdot)^\ff A$ is a map such that :
	\begin{itemize}
		\item If $P\in\mm P$ is an $n$-ary predicate symbol, then 
		$P^\ff A: A^n\to\cuatro$
		\item If $f\in\mm F$ is an $n$-ary function symbol, then 
		$f^\ff A: A^n\to A$
		\item If $c\in\mm C$, then $c^\ff A\in A$
	\end{itemize}
	
\end{defi}

\

\begin{defi}\label{asig-valter}
Given a non-empty set $X$, an {\em assignment} on $X$ is a map $s:\mm V\to X$. If $\ff A$ is a \cuatro-structure over the signature $\mm S$ then an {\em assignment on $\ff A$} is an assignment on the domain $A$ of the \cuatro-structure $\ff A$. We denote by $\ese A$ the set of all assignment on $\ff A$, i.e. $\ese A=A^\mm V$.

Given $s\in \ese A$ and $t\in\ter(\mm S)$, the {\em value of the term} $t$  {\em in $\ff A$ by the assignment $s$} (denoted $t^\ff A[s]$) is defined inductively as follows:
\begin{itemize}
	\item If $t\in\mm V$, then $t^\ff A[s]=s(t)$.
	\item If $t\in\mm C$, then $t^\ff A[s]=t^\ff A$.
	\item If $t$ is $f(t_1,\ldots,t_n)$, where  $f\in\mm F$ and $t_i\in\ter(\mm S)$, then $t^\ff A[s]=f^\ff A\left(t_1^\ff A[s],\dotsc,t_n^\ff A[s]\right)$.
\end{itemize}
\end{defi}

\begin{nota} Given $s\in\ese A$, $x\in\mm V$ and $a\in A$, we denote by $s_x^a$ the assignment over $\ff A$ that satisfies
$$
s_x^a(y) \ = \ 
\left\{
\begin{array}{ll}
	s(y)&\mbox{if } y\neq x\\
	a& \mbox{if } y=x
\end{array}
\right.
$$
\end{nota}

\begin{defi}\label{cube-estruc}
	A $\cube$-structure over $\mm S$ is a pair  
	$$\brac{\ff A,\valbede\cdot }$$
	where $\ff A$ is a \cuatro-structure over $\mm S$ and $\valbede\cdot$ is map defined inductively as follows, for every $s\in\ese A$
	\begin{enumerate}
		\item 
		$\valbede{P(t_1,\dotsc,t_n)}(s) \ = \ P^\ff A\left(t_1^\ff A[s],\dotsc,t_n^\ff A[s]\right)$, where $P\in\mm P$ is an $n$-ary predicate symbol,
		\item 
		$\valbede{\#\varphi}(s) \ = \ \# \valbede\varphi(s)$, where $\#\in\set{\copi,\neg}$,
		\item $\valbede{\varphi\mathop{\#}\psi}(s)=\valbede\varphi(s)\mathop{\#}\valbede\psi(s)$, where $\#\in\set{\wedge,\vee,\imp}$,
		\item $\valbede{\pt x\varphi}(s)=\inf\left(\set{\valbede\varphi(s_x^a) \ : \  a\in A}\right)$,
		\item $\valbede{\ex x\varphi}(s)=\sup\left(\set{\valbede\varphi(s_x^a) \ : \  a\in A}\right)$.
	\end{enumerate}
\end{defi}

\

\begin{nota} \label{DM-infty}
Given that {\bf 4} is a finite lattice, it is complete and it satisfies the following, for every $\{a_i \ : \ i \in I\} \subseteq {\bf 4}$: $\neg \inf_{i \in I} a_i= \sup_{i \in I} \neg a_i$ and  $\neg \sup_{i \in I} a_i= \inf_{i \in I} \neg a_i$. From this,  $\neg \inf_{i \in I} \neg a_i= \sup_{i \in I} a_i$. 
\end{nota}


\begin{defi}\label{modelos} Let $\ff A$ be a $\cube$-structure over $\mm S$ and $\varphi\in\for(\mm S)$.
     We say that $s\in\ese A$ 
     {\em satisfies $\varphi$ in the $\cube$-structure $\ff A$} 
     if 
     $\valbede\varphi(s)\in \set{1,\bebe}=\mm D_\bededos$. 
     We say that $\ff A$ is a  {\em $\cube$-model} of $\varphi$ if every assignment on $\ff A$ satisfies $\varphi$. We denote by ${\rm mod}_\cube (\varphi)$ the class of all $\cube$-models of $\varphi$. Let $\Gamma\subseteq \for(\mm S)$, we say that $\ff A$ is a  {\em \cube-model} for  $\Gamma$ if $\ff A\in\bigcap_{\gamma\in\Gamma}{\rm mod}_\cube(\gamma)$; we denote by ${\rm mod}_\cube(\Gamma)$ the class of all models of $\Gamma$, i.e. ${\rm mod}_\cube(\Gamma)=\bigcap_{\gamma\in\Gamma}{\rm mod}_\cube(\gamma)$.
\end{defi}

\begin{defi}\label{martilloprimer}
	The logic $\cube$ is the pair $\brac{\rr L(\pt,\ex,{\bb P}_T,\mm S),\, \mocube}$ where the consequence operator $\mocube$ is defined as follows
	\begin{center}
		$\Delta\mocube\varphi$ \ iff \ ${\rm mod}_\cube(\Delta)\subseteq{\rm mod}_\cube(\varphi)$ 
	\end{center}
	in this case, we say that $\varphi$ is a {\em $\cube$-consequence} of $\Delta$ and, if $\Delta=\emptyset$, we say that $\varphi$ is {\em $\cube$-valid.}
\end{defi}

\

\begin{teo}\label{martillo-sem} Let $\ff A$ be a $\cube$-estructura and let $s\in\ese A$.
	\
	
	\begin{enumerate}
		\item[] $\valbede{\ex x\varphi}(s)\in\set{1,\bebe}$ implies $\valbede{\varphi}(s_x^b)$ for some $b\in A$
	\end{enumerate}
\end{teo}
\begin{proof}
 It is consequence of Definition~\ref{cube-estruc}(5) and Remark~\ref{DM-infty}.
\end{proof}

\

\begin{propo}\label{def-ex-pt}
	The quantifier $\ex$ is definable from $\pt$ and $\neg$ in $\cube$ as  $\ex x\varphi:=\existe x\varphi$. 
\end{propo}
\begin{proof}
It is consequence of Definition~\ref{cube-estruc}(5) and Remark~\ref{DM-infty}.
\end{proof}

\begin{nota} \label{not-strongeq}
The formulas $\ex x\varphi$ and $\nof\pt x\nof\varphi$ are not strongly equivalent in $\cube$, that is: despite being equivalent, their denotations are not necessarily identical in any structure. To see this, consider the next example: \ consider the first-order signature $\mm S=\brac{\set P,\emptyset,\emptyset}$ where $P$ is an unary predicate symbol and consider the $\cube$-structure $\ff A$  where its domain is $A=\set{a,b}$ and such that  
\begin{align*}
	P^\ff A(a)&=0&
	P^\ff A(b)&=\bebe
\end{align*}
Then $\valbede{\ex xP(x)}=\sup\set{0,\bebe}=\bebe$, but, on the other hand, 
\begin{align*}
\valbede{\nof\pt x\nof\varphi}
&=\nof\left(\inf\set{\nof 0,\nof\bebe}\right)\\
&=\nof\left(\inf\set{1,0}\right)\\
&=\nof 0=1\\
\end{align*}
\end{nota}

\

\section{The deductive system $\hachecube$}

In this section, we shall present a syntactic version of $\cube$ in terms of the Hilbert-style calculus $\hachecube$ which extends $\hachedost$. Let $\mathbb{P}_T=\{\land, \vee, \to, \neg,\copi\}$, then:
\begin{defi}\label{sistema}
Consider the first-order signature $ \mm S=\langle\mm P,\mm F,\mm C\rangle$. The Hilbert-style calculus  $\hachecube$ over the language $\rr L(\pt,{\bb P}_T,\mm S)$ is the extension of  $\hachedost$, expressed in the language $\rr L(\pt,\bb P,\mm S)$), by adding the following:
	\paragraph{Axioms}\hfill
	\begin{description}
		\item[\axi A]   $\pt x\varphi\imp\varphi[x/t],$
		\quad  $t\in \libre{x, \varphi}$   
		\item[\axi B]   $\no\pt x\varphi\imp\no\pt x\no\no\varphi$
		\item[\axi C]   $\no\varphi[x/t]\imp\no\pt x\varphi$
		\quad $t\in \libre{x, \varphi}$  
	\end{description}
	\paragraph{Rules of inference}
	\begin{description}
		\item[($\pt$-{\bf In1})] \ $\displaystyle\frac{\alpha\imp\beta}{\alpha\imp\pt x\beta},$
		\quad 
		 $x\not\in \libre{\alpha}$
		\item[($\pt$-{\bf In2})] \ $\displaystyle\frac{\beta\imp\alpha}{\no\pt x\no \beta\imp\alpha},$
		\quad 
		$x\not\in \libre{\alpha}$
	\end{description}
\end{defi}

\begin{nota} We choose to work with just the universal quantifier. Though (as it was established in the previous section) $\pt$ and $\ex$ are  semantically interdefinable by means of the paraconsistent negation $\no$, that is
$$
\ex x\psi\equiv\no\pt x\no\psi
$$
we were not able to prove axioms such as  \axi B or the rule of introduction of $\ex$.  This is due to fact that the paraconsistent negation $\no$ enjoys less ``good'' properties than the classic negation. On the other hand, the strong negation $\nof$ is classic, that is, it is  explosive and complete; and so, one would be tempted to recover the axioms and rules that involve $\ex$ using $\nof$ instead of $\no$ to define $\ex$. But we have the problem that formulas such as $\ex x\alpha$ and $\nof\pt x\nof\alpha$ are not (strongly) equivalent, as pointed out in Remark~\ref{not-strongeq}.
\end{nota}

\begin{nota}\label{instancia}
Let $\mm S$ be a first-order signature. An {\em instance} of a propositional formula  $\varphi(p_1,\dotsc,p_n)$ in $\rr L(\bb P_0)$ is a formula in $\rr L(\mm S)$ obtained from  $\varphi$ by performing a simultaneous substitution of each ocurrence of the propositional variable $p_i$ by the formula $\beta_i$ in $\rr L(\mm S)$, for $1\leq i\leq n$. Let us denote by $\varphi[\beta_1/p_1,\dotsc,\beta_n/p_n]$ such instance.
\end{nota}

\begin{teo}\label{preserva-instancia}
	If  $\varphi\in\for(\mm S)$ is an instance of a formula $\bededos$-valid, then $\macube\varphi$.
\end{teo}
\begin{proof} It is a direct adaptation of \cite[Proposition 2.1]{men}.
\end{proof}

\subsection{Deduction metatheorem for $\hachecube$ }

\begin{defi}\label{depende}
	Let
	$\Gamma\subseteq \for(\mm S)$, 
	$\varphi\in\Gamma$ 
	and let  $\psi_1\ldots\psi_n$ be a deduction of $\psi_n$ from $\Gamma$ in \hachecube.	
	We say that $\psi_i$ {\em depends  upon $\varphi$} in the deduction if it is verified one of the following conditions:
	\begin{itemize}
		\item[(a)] $\psi_i$ is $\varphi$; 
		\item[(b)] $\psi_i$ is a direct consequence by \emepe, {\bf ($\pt$-In1)}\ or {\bf ($\pt$-In2)} of some preceding
formulas of the sequence, where at least one of these preceding formulas depends upon  $\varphi$.
	\end{itemize}
\end{defi}

The proof of the following result can be done exactly as in the case of first-order classical logic, taking into account that $\hachecube$ is an axiomatic extension of that logic. A detailed proof for the classical case can be found, for instance, in~\cite[Proposition~1.9]{men}.

\begin{propo}[Deduction Metatheorem \ (MTD)]\label{Predmdt}
	If in some deduction showing that $\Gamma,\varphi\vdash\psi$ 
	no application of {\bf ($\pt$-In1)}\ and {\bf ($\pt$-In2)}\ 
	to a formula that depends upon $\varphi$ 
	has as its quantified variable a free variable of $\varphi$, 
	then $\Gamma\vdash\varphi\imp\psi$.
\end{propo}

\begin{coro}\label{corodmt}
Suppose that there is some deduction of $\Gamma,\varphi\vdash\psi$ involving no application of {\bf ($\pt$-In1)}\ and {\bf ($\pt$-In2)}  in which the quantified variable is free in  $\varphi$. Then, $\Gamma\vdash\varphi\imp\psi$.
\end{coro}

\begin{coro}(Deduction Metatheorem for sentences)\label{sendmt}
	If $\varphi$ is a sentence and $\Gamma,\varphi\vdash\psi$, then $\Gamma\vdash\varphi\imp\psi$.
\end{coro}

The following two propositions can be proven as in the case of first-order classical logic and first-order {\bf LFI}s (see~\cite[Chapter 7]{CarCon}).

\begin{propo} The rule of generalization
	\begin{center}
		\gen \ \quad $\varphi\vdash\pt x\varphi$
	\end{center}
is derivable in $\hachecube$.
\end{propo}

\

\begin{propo}\label{teoremas-hachecu}
\

	\begin{itemize}
		\item[(i)] $\alpha\macube\no\pt x\no\alpha(t/x)$
		\item[(ii)] $\nof\pt x\alpha\macube \no\pt x\no\nof\alpha$
		\item[(iii)] $\no\alpha\imp\beta\macube\, \no\pt x\alpha\imp\beta$,   if $x\not\in\libre\beta$
		\item[(iv)] $\macube\pt x\nof\alpha\imp\nof\existe x\alpha$
		\item[(v)] $\nof\no\beta\macube\nof\no\pt x\beta$
		\item[(vi)] $\macube\nof\pt x\nof\alpha\imp\existe x\alpha$
		\item[(vii)] $\macube\left(\alpha\imp\no\pt x\beta\right)\imp \existe x(\alpha\imp \no\beta)$, 

where $\alpha\in\sent(\mm S)$ and $\libre\beta=\set x$
		\item[(viii)]$(\varphi\imp\no\phi)\imp\psi
		\macube 
		(\varphi\imp\no\pt x\phi)\imp\psi$, 

where $\varphi\in\sent(\mm S)$,  $\libre\phi=\set x$ and $x\not\in\libre\psi$.
		
	\end{itemize}
\end{propo}

\begin{lema}\label{lema-instancias}
Let
$\alpha=A[\beta_1/b_1,\ldots,\beta_n/b_n]\in\for(\forall,{\bb P}_T,\mm S)$ 
be an instance of a formula $A(b_1,\ldots,b_n)$ in the propositional language  $\rr L({\bb P}_T)$. 
 For every $s\in S(\ff A)$ we define the $\bededos$-morphism $e_s$ 
as follows
\begin{center}
	$e_s(b_i)\, :=\, k$ \ iff \ $\valbede{\beta_i}(s)=k$,
	
\end{center}\noindent
for every $k\in\set{1,\bebe,\nene,0}$, then $$\valbede{\alpha}(s)=e_s(A).$$
\end{lema}
\begin{proof} 
Using induction on the complexity of $A$.
\end{proof}

\begin{coro}\label{validez-instancias}
If $\varphi$ is an instance of an axiom of $\hachedost$, then $\mocube\varphi$.				\end{coro}

\begin{lema}\label{asig-var-libres}
Let $t(x_1,\ldots,x_n)$ be a term and let $\alpha(x_1,\ldots,x_n)$ be a formula in 
$\rr L(\mm S)$. If $s$ and $s'$ are assignment in $\ff A$ such that  
$s(x_i)=s'(x_i)$ for all $i$ ($1\leqslant i\leqslant n$), then 
\begin{itemize}
\item[(i)] $t^\ff A[s]=t^\ff A[s']$,
\item[(ii)] $\valbede\alpha(s)=\valbede\alpha(s')$,
\end{itemize}
\end{lema}
\begin{proof} See \cite[Proposition 2.2]{mosto}. 
\end{proof}

\begin{lema} [Substitution Lemma] \label{lema-axiomas} 
Let $t\in \libre{x,\varphi}$ and let $s$ be an assignment. Then 
$$\valbede{\varphi(t/x) }(s) = \valbede\varphi \left(s_x^{a}\right)$$
where $a=t^\ff A[s]$
\end{lema}
\begin{proof} It is routine. 
\end{proof}

\

\begin{lema}\label{validez-axiomas}
Let $t\in \libre{x,\varphi}$, then \\[3mm]
(a) \ \ $\mocube\pt x \varphi\imp\varphi(t/x)$, \\[1mm]
(b) \ \ $\mocube\no\pt x\varphi\imp\no\pt x\no\no\varphi$, \\[1mm]
(c) \ \ $\mocube	\no\varphi(x/t)\imp \no\pt x\varphi$.
\end{lema}
\begin{proof} \ \\
(a) Suppose that $\valbede{\pt x\varphi}(s) = \inf_{a \in A} \valbede\varphi(s_x^a)\in\set{1,\bebe}$. By definition of the order in {\bf 4}, $\valbede\varphi(s_x^a)\in\set{1,\bebe}$ for every $a$. In particular, $\valbede\varphi(s_x^b) = \valbede{\varphi[t/x]}(s)\in\set{1,\bebe}$ for   $b=t^\ff A[s]$, by Lemma~\ref{lema-axiomas}.\\[1mm]
(b)  It is immediate from the fact that  $\valbede{\varphi}(s)=\valbede{\no\no\varphi}(s)$ for every $s\in \ese A$.\\[1mm]
(c) Suppose that  $\valbede{\no\varphi(x/t)}(s)\in\set{\bebe,1}$. Then, $\valbede{\neg\varphi}(s_x^b)\in\set{\bebe,1}$  for   $b=t^\ff A[s]$, by Lemma~\ref{lema-axiomas}. But then $\valbede\varphi(s_x^b)\in\set{\bebe,0}$ and so $\valbede{\pt x\varphi}(s) = \inf_{a \in A} \valbede\varphi(s_x^a)\in\set{\bebe,0}$, by definition of the order in {\bf 4}. Hence $\valbede{\no\pt x\varphi}(s)\in\set{\bebe,1}$. 
\end{proof}

\begin{lema}\label{validez-reglas} \
	
\begin{itemize}
\item[(1)] $\varphi\imp\psi,\, \varphi \, \mocube \, \psi$
\item[(2)] If $x \notin \libre{\varphi}$, then  
\begin{center}
	(2a) \ $\varphi\imp\psi\mocube\varphi\imp\pt x \psi$
	\quad and \quad 
	(2b) $\psi\imp\varphi\mocube\existe x\psi\imp\varphi$ \
\end{center}
\end{itemize}
\end{lema}
\begin{proof} 
{\em (1)} Analogous to the proof done for the propositional case. \\[1mm]
{\em (2a)} \ Suppose that $\ff A\in{\sf mod}_\cube(\varphi\imp\psi)$. Let us pick $s\in \ese A$, then $\valbede{\varphi}(s)\in\set{1,\bebe}$. Since $x\notin\libre{\varphi}$, for all $a\in A$, we have $\valbede\varphi(s_x^a)\in\set{1,\bebe}$ (by  Lemma \ref{asig-var-libres}). Therefore, for all $a\in A$ we have that $\valbede\psi(s_x^a)\in\set{1,\bebe}$. Hence, $\valbede{\pt x\psi}(s)\in\set{1,\bebe}$ and then $\ff A\in{\sf mod}_\cube(\varphi\imp\pt x\psi)$.\\[1mm]
{\em (2b)} \ Suppose that $\ff A\in{\sf mod}_\cube(\psi\imp\varphi)$. Let $s\in A^{\mm V}$ an arbitrary assignment such that $\valbede{\existe x\psi}(s)\in\set{1,\bebe}$. Then, by Theorem~\ref{martillo-sem} and Proposition~\ref{def-ex-pt}, there is some $a\in A$ it holds  $\valbede{\psi}(s_x^a)\in\set{1,\bebe}$. Then, $\valbede\varphi(s_x^a)\in\set{1,\bebe}$. By Lemma \ref{asig-var-libres}, since $x\notin \libre{\varphi}$ we have $\valbede\varphi(s)\in\set{1,\bebe}$. Then,  $\ff A\in{\sf mod}_\cube(\existe x\psi\imp\varphi)$.
\end{proof}

\

\begin{teo} [Soundness] \label{sound-qbd2}
If  $\Gamma\macube\varphi$ \ then  \ $\Gamma\mocube\varphi$.
\end{teo}
\begin{proof} 
By induction on the length of the derivation of  $\varphi$ from $\Gamma$ in $\hachecube$, using Lemmas \ref{validez-axiomas} and \ref{validez-reglas}.
\end{proof}

\section{Scapegoat sets of formulas and completeness}

In this section, as it is usual in the context of first-order logics, we call {\em theory} to any set of closed formulas.

\begin{lema}\label{lema-sat-hq}
	Let $\varphi\in\sent(\mm S)$ and let  $\Delta\subseteq \for(\mm S)$ be a $\varphi$-saturated in $\hachecube$, then 
	\begin{center}
		either \ {\rm $\no\varphi\wedge\copi\varphi\in\Delta$}, \ 
		or \ {\rm $\no\copi\varphi\wedge\copi\varphi\in\Delta$}
	\end{center}
\end{lema}

\begin{proof}
	Necessarily, it holds that $\copi\varphi\in\Delta$. Otherwise, since $\Delta$ is a $\varphi$-saturated theory, we have 
	$$
	\Delta,\, \copi\varphi \macube \varphi
	$$
	and, in virtue of Proposition~\ref{esq-validos}(i), (DMT) (Corollary~\ref{sendmt}) and MP we conclude
	$$
	\Delta\macube \varphi
	$$
	which is a contradiction since $\Delta$ is $\varphi$-saturated.
	
	On the other hand, suppose that $\no\varphi\not\in\Delta$ \ and \ $\no\copi\varphi\not\in\Delta$. 
	Since $\Delta$ is $\varphi$-saturated we have  
	\begin{center}
		$\Delta,\, \no\varphi\macube \varphi$ \ and \ 
		$\Delta,\, \no\copi\varphi\macube \varphi$  
	\end{center}
	Using Corollary~\ref{sendmt} once again and axiom \pos 5, we obtain 
	$$
	\Delta\macube (\no\varphi\imp\varphi)\wedge(\no\copi\varphi\imp\varphi)
	$$
	and by Proposition~\ref{esq-validos}(ii), it follows that  $\Delta\macube\varphi$, 
	which is a contradiction.
	
	Finally, both conditions cannot hold simultaneously since, by Proposition \ref{esq-validos}  items~(iii), (vi) and~(vii), $\Delta$ would be a trivial theory, which is a contradiction.
\end{proof}

\begin{defi}
A theory $\Gamma$ is said to be {\em maximal} in $\hachecube$ if it is $\bot$-saturated for some formula $\bot$ such that $\bot \macube \beta$ for every $\beta$.
\end{defi}

So, $\Gamma$ is maximal iff $\Gamma$ is not trivial, but $\Gamma \cup \{\beta\}$ is trivial for any $\beta \not\in \Gamma$.

\begin{propo} \label{satu=max}
Let $\Gamma$ be a $\varphi$-saturated set in $\hachecube$. Then, it is maximal in $\hachecube$.
\end{propo}
\begin{proof}
Observe that $\Gamma \not\macube \bot$: otherwise $\Gamma \macube \varphi$, a contradiction. By Theorem~\ref{L-A-lemma}, there exists a $\bot$-saturated set $\Delta$ (that is, a maximal theory) which contains $\Gamma$. Observe that $\varphi \not\in \Delta$: otherwise, since either $\no\varphi\wedge\copi\varphi\in\Gamma$ or  $\no\copi\varphi\wedge\copi\varphi\in\Gamma$, by Lemma~\ref{lema-sat-hq}, it would follow that either  $\bot_\varphi:=\varphi \land \no\varphi\wedge\copi\varphi\in\Delta$ or  $\bot'_\varphi:=\varphi \land\no\copi\varphi\wedge\copi\varphi\in\Delta$. But both formulas $\bot_\varphi$ and $\bot'_\varphi$ are bottom in  $\hachecube$ by axiom {\bf ($\copi$-1)} and Proposition~\ref{esq-validos}(iii), and so $\Delta$ would be trivial, a contradiction. Hence, $\varphi \notin \Delta$.

Suppose now that there exists $\beta \in \Delta$ such that $\beta \not\in \Gamma$. Then, $\Gamma,\beta \macube \varphi$. But then, $\Delta \macube\varphi$ and so $\varphi \in \Delta$, a contradiction. This shows that $\Delta=\Gamma$ and so $\Gamma$ is maximal in $\hachecube$.
\end{proof}

Since maximal theories are deductively closed, we can prove without difficulty the next result.
\begin{lema}\label{maximas-lema}
	If $\Delta \subseteq \for(\mm S)$ is  maximal in $\hachecube$, then
	\begin{enumerate}[(a)]
		\item If $\set{\alpha\imp\beta,\, \alpha}\subseteq\Delta$, then $\beta\in\Delta$.
		\item $\alpha\wedge\beta\in\Delta$ \ iff \ $\set{\alpha,\beta}\subseteq\Delta$.
                     \item $\no(\alpha\wedge\beta)\in \Delta$ \ iff \ $\no\alpha\in\Delta$ \ or \ $\no\beta\in\Delta$.
		\item If $\pt x\alpha \in \Delta$, then $\alpha[t/x]\in\Delta$ for every $t\in{\sf clo}(\mm S)$.
                     \item If $\no\alpha[t/x]\in\Delta$ for some $t\in{\sf clo}(\mm S)$, then $\no\pt x\alpha\in\Delta$.
	\end{enumerate}
\end{lema}

The next result will be needed for defining canonical models.

\begin{teo}\label{well-def-hq} Let $\Delta\subseteq\for(\mm S)$ be maximal in \hachecube. Then, for all  $\varphi\in\sent(\mm S)$ it holds one, and only one, of the following conditions:
		\begin{align*}
			(1)&\ \varphi\wedge\copi\varphi\in\Delta&
			(2)&\ \varphi\wedge\no\varphi\in\Delta
			\\
			(3)&\ \copi\varphi\wedge\no\copi\varphi\in\Delta&
			(4)&\ \no\varphi\wedge\copi\varphi\in\Delta
			\\
		\end{align*}
\end{teo}

\begin{proof} Let $\alpha\in\for$ be an arbitrary formula. Then, $\alpha \in \Delta$ or $\alpha \not\in \Delta$.
	
	\paragraph{Case 1} If $\alpha\in \Delta$, by Proposition~\ref{esq-validos}(iii) and the maximality of $\Delta$, necessarily  
	$$
	\copi\alpha\wedge\no\copi\alpha \ \not\in \ \Delta
	$$
	Then, by Lemma \ref{maximas-lema}(b), one of the formulas $\copi\alpha$, \, $\no\copi\alpha$ does not belong to $\Delta$. 
	\paragraph{Case 1.1} If $\no\copi\alpha\not\in\Delta$, by the maximality of $\Delta$ we have that 
	$$
	\Delta,\, \no\copi\alpha\macube\copi\alpha
	$$
	and, by Proposition~\ref{esq-validos}(iv) it holds 
	$$\Delta\macube\copi\alpha$$
	Then, 
	$$\ \alpha\wedge\copi\alpha\in\Delta$$
	\paragraph{Case 1.2} If $\copi\alpha\not\in\Delta$, again, by the maximality of $\Delta$ it holds
	$$
	\Delta,\, \copi\alpha\macube\no\alpha
	$$
	and by Proposition \ref{esq-validos}(v)
	$$\Delta\macube\no\alpha$$
	that is, since $\Delta$ is closed, $\no\alpha\in\Delta$. Then,
	$$ \ \alpha\wedge\no\alpha\in\Delta$$

	\paragraph{Case 2}  $\alpha\not\in\Delta$. Then $\Delta$  is $\alpha$-saturated. Therefore, by Lemma \ref{lema-sat-hq}
	\begin{center}
		\ (3) $\no\alpha\wedge\copi\alpha\in \Delta$ \  or \
		\ (4) $\no\copi\alpha\wedge\copi\alpha\in \Delta$.
	\end{center}
	
	Finally, let us note that conditions (1)--(4) cannot hold simultaneously, moreover, they are mutually exclusive and this is consequence of Proposition  \ref{esq-validos}(iii) and axiom {\bf (\copi-1)}.
\end{proof}

\begin{defi}\label{def-testigos} Let $\Delta$ be an arbitrary set of formulas in the language $\rr L(\pt,{\bb P}_T,\mm S)$ and let $C$ be a nonempty set of constant symbols in the signature $\mm S$. We say that  $\Delta$ has witnesses in $C$ (or that it is a Henkin set)  for $\hachecube$ if it holds: 
\begin{quote}
for every sentence of the form  $\no\pt x\varphi$, there is a constant symbol $c\in C$ such that 
if $\Delta\mabedos\no\pt x\varphi$, then $\Delta\mabedos\no\varphi(c)$
\end{quote}
\end{defi}

\begin{teo}[Theorem of Constants]\label{const-conse}
Let $\Delta$ be an arbitrary set of formulas in the language $\rr L(\pt,{\bb P}_T,\mm S)$ and let $\sststile{\hachecube} C$ be the consequence relation of $\hachecube$ on the signature $\mm S_C$, which is obtained from $\mm S$ by adding the new constant symbols of $C$. Then, for every $\varphi\in\for(\mm S)$, 
\begin{center}
$\Delta\mabedos\varphi$ \ iff \ $\Delta\sststile\hachecube C\varphi$
\end{center}
That is, $\hachecube$ over $\mm S_C$ is a conservative extension of  $\hachecube$ over $\mm S$
\end{teo}
\begin{proof}
Analogous to the proof of  \cite[Theorem 7.5.2]{CarCon}.
\end{proof}

\

\begin{teo}\label{teo-testigos}
Let $\Delta\subseteq\for(\mm S)$. Then, there exists $\Delta^W\subseteq\for(\mm S)$ with witnesses in the set $C$ for $\hachecube$ such that $\Delta \subseteq \Delta^W$ and
\begin{center}
$\Delta\macube\varphi$ \ iff \ $\Delta^W\sststile\hachecube C\varphi$
\end{center}
Besides, any extension of $\Delta^W$ by sentences in the signature $\mm S_C$ is a set with witnesses in  $C$.
\end{teo}
\begin{proof}
Let $\Delta$ be a set of formulas in $\rr L(\mm S)$ and let $C$ be a set of new constant symbols such that $|C|=\|\rr L(\mm S)\|$. 
Let
\begin{center}
$\psi_0(x_{i_0}),\psi_1(x_{i_1}),\psi_2(x_{i_2}),\ldots,\psi_k(x_{i_k}),\ldots$ 
\end{center}
be an enumeration of all formulas with one free variable in the language $\rr L(\mm S_C)$. We choose a sequence of elements in $C$
\begin{center}
$a_0,\ldots, a_n,\ldots$
\end{center}
in such a way that
\begin{itemize}
\item $a_k$ does not occur in the formulas  
$\psi_0(x_{i_0}),\psi_1(x_{i_1}),\ldots,\psi_{k-1}(x_{i_{k-1}})$, and
\item each $a_k$ is different from $a_0,a_1,\ldots,a_{k-1}$.
\end{itemize}
Now, consider the following sentences:
\begin{center}
$(S_k)\quad 
\no\pt x_{i_k}\psi_k(x_{i_k})
\imp
\no\psi_k(a_k)$
\end{center}

By construction, we can assert that each of the new symbols $a_k$ occurs only in $(S_k)$.
Let $\Delta^W \ := \ \Delta\cup\{(S_i)\}_{i\in\omega}$. By construction, we have that $\Delta\subseteq\Delta^W\subseteq\for(\mm S_C)$ and $\Delta^W$ has witnesses in $C$ (see  Definition \ref{def-testigos}). Let us see that  $\Delta^W$ is a conservative extension of $\Delta$. Indeed, let $\varphi\in\for(\mm S)$ and suppose that  $\Delta^W\macuhen\varphi$. Since $\hachecube$ is finitary, we can assert that there exists a finite set $\Delta_0\subseteq \Delta^W$ such that $\Delta_0\macuhen\varphi$. Besides, $\Delta_0$ has a finite number of formulas $(S_k)$.

\

 Suppose that the formula $(S_{n})$ is in $\Delta_0$ and let
 $\Delta_1 \ :=  \ \Delta_0 \setminus \set{\no\pt x_{i_n}{\psi_n}\imp\no\psi_n(a_{n})}$. By Corollary \ref{sendmt}, \ 
$\Delta_1\macuhen (\no\pt x_{i_n}\psi_{n}\imp\no\psi(a_{n}))\imp\varphi$.

\

By an usual technique that can be found in the literature (for instance \cite[Theorem 7.5.2]{CarCon}), we can built a deduction
$$\Delta_1\macuhen (\no\pt x_{i_n}\psi_{n}\imp\no\psi_{n}(y))\imp\varphi$$
where $y$ is a variable symbol which does not occur in $\no\pt {x_{i_n}}{\psi_n}\imp\no\psi_n(a_n)\imp\varphi$.
\item From Proposition \ref{teoremas-hachecu}(viii), it follows
$$\Delta_1\macuhen\big(\no\pt x_{i_n}\psi_{n}(x_{i_n})\imp\no\pt y\psi_{n}(y)\big)\imp\varphi$$
By \axi C, we have
$\macuhen\no\psi_{i_n}(x_{i_n})\imp\no\pt y\psi_{i_n}(y)$, 
and by Proposition \ref{teoremas-hachecu}(iii), 
$\macuhen \no\pt x_{i_n}{\psi_n}(x_{i_n})\imp{\no\pt y}{\psi_n}(y)$. Therefore
$$\Delta_1\macuhen\varphi$$
Repeating this procedure a finite number of times (since $\Delta_0$ is finite) we conclude
	$$\Delta_n\macuhen\varphi,$$ 
	where $\Delta_n$ is a theory without the axioms $(S_k)$. Then,  $\Delta\macuhen\varphi$ and by Theorem \ref{const-conse}, we have $\Delta\macube\varphi$. Then, $\Delta^W$ is a conservative extension of $\Delta$.
\end{proof}

\begin{defi} \label{maxtheory}
 We say that the set of sentences  $\Theta$ in the language $\rr L(\mm S)$ is a  {\em maximal theory for $\hachecube$} it there is a maximal theory $\Delta\subseteq \for(\mm S)$ in $\hachecube$ such that
	$$\Theta=\Delta\cap\sent(\mm S)$$
\end{defi}

\begin{teo}[Henkin's model existence]\label{henbebd}
Let $\Theta$ be a maximal theory and with witnesses for $\hachecube$. Then, $\Theta$ has a  \cube-model.
\end{teo}
\begin{proof} 
We define the  \cube-structure $\ff A$ on $\mm S$ as follows:
\begin{itemize}
\item 
the domain $A$ is the set of all closed terms in $\mm S$, ${\sf clo}(\mm S)$
\item 
if $c$ is a constant symbol of $\mm S$, \
$c^\ff A:=c$,
\item 
if $f$ is an $n$-ary function symbol of $\mm S$, $f^\ff A:A^n\imp A$ is given by  $f^\ff A(\vect a):=f(\vect a)$,
\item if $R$ is an $n$-ary predicate symbol of $\mm S$, $R^\ff A$ is given as follows 
\begin{align*}
R^\ff A(\vect a)=1 
&& \ \mbox{ iff } \ &&
\bduno{R(\vect a)}\in \Theta \\
R^\ff A(\vect a)=\bebe 
&& \ \mbox{ iff } \ &&
\bdebe{R(\vect a)}\in \Theta \\
R^\ff A(\vect a)=\nene 
&& \ \mbox{ iff } \ &&
\bdene{R(\vect a)}\in \Theta \\
R^\ff A(\vect a)=0 
&& \ \mbox{ iff } \ &&
\bdcero{R(\vect a)}\in \Theta 
\end{align*}
\end{itemize}
Note that the predicates are well defined in virtue of Theorem~\ref{well-def-hq}.
Then, for every sentence $\varphi\in\sent(\mm S)$, the  \cube-structure $\ff A$ satisfies the following property (P)
\begin{align*}
&(1)\quad
\valbede\varphi= 1
&\ \mbox{ iff } \ &&
\bduno\varphi\in \Theta
\\
&(2)\quad
\valbede\varphi= \bebe
&\ \mbox{ iff } \ &&
\bdebe\varphi\in \Theta
\\
&(3)\quad
\valbede\varphi= \nene
&\ \mbox{ iff } \ &&
\bdene\varphi\in \Theta\\
&(4)\quad
\valbede\varphi= 0
&\ \mbox{ iff } \ &&
\bdcero\varphi\in \Theta
\end{align*}
Indeed, we use induction on the complexity of the sentence $\varphi$. We show only the ``if'' part of conditions (1)--(4) (we left the ``only if'' to the reader).
\paragraph{Base step.} 
If $\varphi$ is the atomic formula $R(\vect a)$, (P) holds by the definition of the \cube-structure $\ff A$ and the fact that $\valbede{R(\vect a)}=R^\ff A(\vect a)$.

\paragraph{Inductive step.} Let $\varphi$ a sentence of complexity $n\geqq 1$.
Then, we consider the following subcases:
\begin{center}
\begin{tabular}{ll}
{\bf (a)} \ $\varphi$ is $\no\psi$&
{\bf (b)} \ $\varphi$ is $\copi\psi$\\[2mm]
{\bf (c)} \ $\varphi$ is $\psi\wedge\chi$&
{\bf (d)} \ $\varphi$ is $\pt x\psi(x)$
\end{tabular}
\end{center}
We shall just analyze the subcase {\bf (d)}. The others are left to the patient reader.
\paragraph{Case d:} $\varphi$ is $\pt  x\psi$ \\[3mm]
{\bf (d.1)} \ 
If $\bduno{\pt x\psi}\in \Theta$, 
then, by Lemma \ref{maximas-lema}(b) and (d), we have $\psi(a)\in\Theta$ for all $a\in A$. By (I.H.),  $\valbede{\psi(a)}\in\set{1,\bebe}$ for all $a\in A$. Suppose that there is $b\in A$ such that $\valbede{\psi(b)}=\bebe$. Then, by (I.H.), $\no\psi(b)\in\Theta$ and then, by  \axi C and Lemma \ref{maximas-lema}(iii), we get $\no\pt x\psi\in\Theta$. This contradicts the fact that $\Theta$ is non trivial  (recall Proposition~\ref{esq-validos}(iii)), and therefore $\valbede{\pt x\psi}=1$.
\\[2mm]
{\bf (d.2)} If $\bdebe{\pt x\psi}\in\Theta$ then, by Lemma \ref{maximas-lema}(a) and (b), we know that 
\begin{center}
	$\bebe\leq \valbede{\psi(a)}\leq 1$, for all $a\in A$ (I.H.).
\end{center} 
On the other hand, from Lemma \ref{maximas-lema}(b) it follows that $\no\pt x\psi\in\Theta$
and then, there exists a witness $b\in A$ such that $\no\psi(b)\in\Theta$. 
Now, by (I.H.), we have that  $\valbede{\psi(b)}\in\set{0,\bebe}$. 
Then,
\begin{center}
	$\valbede{\psi(b)}=\bebe.$ 
\end{center}
By definition of the order in {\bf 4}, we have $\valbede{\pt x\psi}=\bebe$.
\\[2mm]
{\bf (d.3)} Suppose that $\bdene{\pt x\psi}\in\Theta$. Since  $\mobedos(\copi\alpha\wedge\no\copi\alpha)\imp\nof\alpha$, as it is easy to check,  then $\nof\pt x\psi\in\Theta$, by   Lemma \ref{maximas-lema}(a).
Then, by Proposition~\ref{teoremas-hachecu}(ii) and the fact that $\Theta$ is closed, we have  $\no\pt x\no\nof\psi\in\Theta$, and therefore, there exists $b\in A$ such that $\no\no\nof\psi(b)\in\Theta$. Hence, by {\bf (DNeg)}, $\nof\psi(b)\in \Theta$. Then, by Lemma \ref{maximas-lema}(b) and (c), it holds one of the following cases: 
either $\bdcero{\psi(b)}\in \Theta$ or 
$\bdene{\psi(b)}\in\Theta$.
By  (I.H.),  $\valbede{\psi(b)}\in\set{\nene,0}$. Now, suppose that $\valbede{\psi(c)}\in\set{0,\bebe}$ for some $c\in A$. By (I.H.), we have  $\no\psi(c)\in\Theta$,
and then, by \axi C, we have $\no\pt x\psi\in\Theta$. 
Therefore $\Theta$ is trivial, taking into account that $(\copi{\alpha} \land \neg\copi{\alpha} \land \neg \alpha) \to \beta$ is provable in $\hachedost$, by Proposition~\ref{esq-validos}. From this contradiction we conclude that  
\begin{center}
	$\valbede{\psi(a)}\in\set{1,\nene}$ for all $a\in A$
\end{center}
and then
\begin{center}
	$\valbede{\psi(b)}=\nene$.
\end{center}
Finally, by definition of the order in {\bf 4}, $\valbede{\pt x\psi}=\nene$.
\\[2mm]
{\bf (d.4)} The proof of this case is analogous to case {\bf (d.3)}. \\[2mm]
Therefore, we have proved the property (P).

Let $\theta\in\Theta$. Since $\Theta$ is a maximal theory 
we have that either it holds 
$\theta\wedge\copi\theta\in\Theta$ or it holds $\theta\wedge\no\theta\in\Theta$. By the property (P), we conclude that $\valbede\theta\in\set{1,\bebe}$.
Therefore, $\ff A$ is a model of $\Theta$.
\end{proof}

\begin{teo}[Completeness for sentences]\label{comp-sent-hbd2}
Let $\Delta\cup\set{\varphi}\subseteq \sent(\mm S)$, then
\begin{center}
$\Delta\mocube\varphi$ \ if and only if \ $\Delta\macube\varphi$
\end{center}
\end{teo}
\begin{proof}
Suppose that $\Delta\not\macube\varphi$.  
By Theorem \ref{teo-testigos}, we can extend (conservatively) $\Delta$ to a set $\Delta^W\subseteq\for(\mm S_C)$ with witnesses in $C\neq\emptyset$.  Since $\Delta^W\not\macube\varphi$,
by Theorem~\ref{L-A-lemma} and Proposition~\ref{satu=max}, we know that there exists  
$\Phi\subseteq \for(\mm S_C)$ maximal in  $\hachecube$ such that  $\varphi \not\in \Phi$.
Hence $\Theta=\Phi\cap\sent(\mm S)$ is a  maximal theory for $\hachecube$ (recall Definition~\ref{maxtheory}) such that  $\varphi \not\in \Phi$, which extends $\Delta$. Moreover, $\Theta$ has witnesses.
Then, by Theorem \ref{henbebd}, we know that there is a $\hachecube$-model of $\Theta$. On the other hand, by Definition \ref{well-def-hq}, it must hold one of the following conditions
\begin{align*}
	\bdene{\varphi}&\in \Theta&
	\bdcero{\varphi}&\in \Theta
\end{align*}
and then
$$
\valbede{\varphi}=\set{\nene,0}.
$$
Since $\ff A$ is also a model of $\Delta$, it is not the case that $\Delta\mocube\varphi$.
\end{proof}

\begin{nota} \label{obsquant} 
Observe that completeness was proved only for sentences, while soundness was stated for formulas in general (recall Theorem~\ref{sound-qbd2}). However, a completeness theorems for formulas in general (i.e., for formulas possibly having free variables) can be obtained from Theorem~\ref{comp-sent-hbd2} by observing the following: for any formula $\psi$ let $(\forall)\psi$ be the universal closure of $\psi$, that is: if $\psi$ is a sentence then  $(\forall)\psi=\psi$, and if $\psi$ has exactly the variables $x_1,\ldots,x_n$ occurring free then  $(\forall)\psi=(\forall x_1)\cdots(\forall x_n)\psi$. If $\Gamma$ is a set of formulas then  $(\forall)\Gamma := \{(\forall)\psi  \ : \  \psi \in \Gamma\}$. Thus, it is easy to see that,   for every set $\Gamma \cup \{\varphi\}$ of formulas:  $\Gamma\mocube \varphi$ \ iff \   $(\forall)\Gamma\macube (\forall)\varphi$, and $\Gamma \mocube\varphi$ \ iff \   $(\forall)\Gamma \macube (\forall)\varphi$. From this, a general completeness result follows (see Corollary below).
\end{nota}

\begin{coro}[Adequacy theorem]\label{adecuacion-hbd2} \  $\Delta\mocube\varphi$ \ iff  \ $\Delta\macube\varphi$.
\end{coro}
\begin{proof} The {\em ``if''} part is precisely Theorem \ref{sound-qbd2}. The proof for the {\em ``only if''} part goes as follows:
	\begin{center}
		\begin{tabular}{rll|l}
		$\Delta\mocube\varphi$ 
		& only if & $(\pt)\Delta\mocube(\pt)\varphi$&Remark \ref{obsquant}\\[2mm]
		& only if &
		$(\pt)\Delta\macube(\pt)\varphi$& Theorem \ref{comp-sent-hbd2}\\[2mm]
		& only & 
		$\Delta\macube\varphi$ & Remark \ref{obsquant}		
	\end{tabular}
	\end{center}
\end{proof}

\begin{teo}[Compactness] Suppose that $\Delta\mocube\varphi$ then there exists a finite set  $\Delta_0\subseteq \Delta$ such that $\Delta_0\mocube\varphi$.
\end{teo}

\begin{proof} The usual proof once we have established the adequacy of a first-order semantics w.r.t. a compact proof system. Our deductive system {\em à la Hilbert} $\hachecube$ is compact by definition.
\end{proof}

\end{document}